\documentclass[a4paper,11pt]{article}

\usepackage{amssymb}
\usepackage{art-familyHR3D}
\usepackage{geometry}
\title{A family of three-dimensional Virtual Elements for Hellinger-Reissner elasticity
	problems}

\author{Michele Visinoni\thanks{Department of Mathematics, University of Milan,
	Via Saldini 50, 20133 Milano, Italy
	(michele.visinoni@unimi.it)}}
\date{}

\begin{document}


%
%
	\maketitle	
	\begin{abstract}
		We present a family of Virtual Element Methods for three-dimensional linear elasticity problems based on the Hellinger-Reissner variational principle. A convergence and stability analysis is developed. Moreover, using the hybridization technique and exploiting the information derived from this procedure, we show how to compute a better approximation for the displacement field. The numerical experiments confirm the theoretical predictions.
	\end{abstract}

	\section{Introduction}
The Virtual Element Method (VEM), introduced in~\cite{volley}, is a technology for the approximation of Partial Differential Equations (PDEs) on polytopal meshes, which shares the same variational background of the Finite Element Method (FEM). The core idea behind VEM consists of using local approximation spaces, whose functions are solutions to suitable differential problems. This definition of the discrete spaces provides great flexibility, despite the loss of the explicit knowledge of the discrete functions. Nevertheless, with the information (degrees of freedom) that one has at disposal on the element boundary and interior, it is always possible to compute the discrete bilinear forms and solve the linear system of the problem. 
The design of the Virtual Element Method leads to certain benefits.
The first advantage is related to the opportunity to preserve, at the discrete level, some important features of the continuous problems. For instance, in the elasticity problems~\cite{ARTIOLI2017155,ARTIOLI2018978,DLV}, it is possible to consider a-priori stress tensors without losing the regularity of the solution, in Cahn-Hilliard equation~\cite{ABSVCahnHilliard,ASVVCahnHilliard}, we can take discrete solutions with high-regularity and in Stokes problem, it is possible to preserve the divergence-free property for velocity~\cite{BLV}.
The second advantage regards its robustness in treating general polygonal/polyhedral meshes, including hanging nodes, small edges/faces, and distorted or non-convex elements, allowing us to easily handle, for instance, fractures and contact problems~\cite{BENEDETTO2014135,BENEDETTO201755,CIHAN2022115385,wriggers2017}, but also issues related to adaptivity~\cite{BCNVV_adaptiveVEM}.

%
In these years, VEM has aroused considerable interest in both the mathematical and the engineering community. Here, we only mention, as a representative non-exhaustive sample a brief list of paper~\cite{ABLS_part_I,ABLS_part_II,ARTIOLI2020112667, ARTIOLI_RICOVERYVEM,MORA2020112687,DALTRI2021113663,DASSI20221,Hudobivnik2019,Lamperti2022}.

%
In the present work, we extend the study presented in~\cite{ARTIOLI2018978} to a three-dimensional case. More precisely, we design, analyze and implement conforming VE schemes of general order $k$ for linear elasticity problems. We consider the mixed variational formulation based on the Hellinger-Reissner principle.
As it is well known, imposing both the symmetry of the stresses and the continuity of the tractions at the inter-elements is typically a great source of trouble in the framework of the classical Finite Element Method. 
For these reasons, one usually prefers to relax the continuity condition, considering non-conforming schemes as in~\cite{ArnoldWintherNonconforming} or relax the symmetry of the stress tensor, as in~\cite{Arnold2007mixed}, or change the approach as in~\cite{Schoeberl1,Schoeberl2}.
Here, we want to use the flexibility of VEM to avoid these drawbacks and to design and implement valid alternative schemes, which provide symmetric stresses, continuous tractions and are reasonably cheap concerning the delivered accuracy. 
Furthermore, as in the lowest case~\cite{DLV2021}, the proposed schemes do not have nodal stress degrees of freedom and so we can apply the hybridization strategy to solve the resulting linear system in an efficient way, with also the possibility to construct a post-processed displacement approximation of higher accuracy.
The paper is organized as follows. In Section~\ref{section:modelProblem}, we present the mixed Hellinger-reissner formulation of the 3D linear elasticity problem. Section~\ref{section:VEMHRFAMILY} describes the discrete schemes, while Section~\ref{section:StabilityConvergenceAnalysis} shows the convergence and stability analysis. In Section~\ref{section:hybridization_postprocessing} we briefly present the idea of the hybridization procedure, showing how we construct a better discrete solution for the displacement field through a post-processing procedure. The numerical experiments, that assess the theoretical predictions, are detailed in Section~\ref{section:numerical_results}.
\begin{paragraph}{Notation}
In this paper we will use standard notations for Sobolev spaces, norm and seminorms~\cite{Lions-Magenes}. Give two positive quantities $a$ and $b$, we write $a\lesssim b$ is there exists a positive constants $C$, independent of $a$, $b$ and the mesh size, such that $a\leq Cb$. Moreover, given any set $D\subseteq \R^d$, $d=1,2,3$, and an integer $k\geq 0$, we denote by $\Poly[k]{D}$ the polynomial space up degree $k$, defined on $D$. We denote by $\pi_{k,D}$ the dimension of $\Poly[k]{D}$.
\end{paragraph}
\section{The linear elasticity problem}~\label{section:modelProblem}
We introduce the elasticity problem whose variational formulation is based on the Hellinger-Reissner principle, see~\cite{BoffiBrezziFortin,Braess:book} for more details.
Let $\O\subseteq\R^3$ be a polyhedral domain and we define by $\partial \Omega$ its boundary. To impose suitable boundary conditions, we divide $\partial \Omega$ into two regular disjoint parts $\partial \Omega_D$ and $\partial \Omega_N$, where in the first one we set the natural boundary conditions, while in the second one, the essential conditions. For simplicity, namely for the solvability issues, we suppose that $\partial \Omega_D\neq \emptyset$. Then, the linear elasticity problem reads
\begin{equation}\label{problem:cont-strong}
\left\lbrace{
	\begin{aligned}
		&\mbox{find } (\bfsigma,\bbu)~\mbox{such that}\\
	&-\bdiv \bfsigma= \bbf \quad&\mbox{in $\Omega$}, \\
	& \bfsigma = \C \teps(\bbu)\quad&\mbox{in $\Omega$},
	\end{aligned}	
} \right.
\end{equation}
supplied with the following boundary conditions:
\begin{equation}\label{problem:boundaryCondition}
\left\lbrace{
	\begin{aligned}
	&\bbu=\bbg\ &\mbox{in $\partial\Omega_D$},\\
	&\bfsigma \bbn =\bfpsi\ &\mbox{in $\partial\Omega_N$}.
	\end{aligned}	
} \right.
\end{equation}
Here above, $\bfsigma$ and $\bbu$ represent the stress and the displacement field, respectively. Moreover, $\bbf\in \LTwoTriple{E} $ represents the loading term, $\C$ is the elasticity tensor and $\teps(\cdot)$ is the symmetric gradient operator. 
%
%
Before introducing the weak formulation of Problem~\eqref{problem:cont-strong}, we fix the following notations:
\begin{equation}\label{eq:globalspaces}
U:=\left[L^2(\Omega)\right]^3, \quad
    \Sigma:=\left\{ \bftau\in H(\bdiv;\Omega)\ :\ \bftau \mbox{ is symmetric,}\ \  \bftau\bbn_{|\partial \Omega_N}\right\}
\end{equation}
equipped with standard norms
\begin{equation*}
    \norm[U]{\bbu}^2:=\int_{\O}|\bbu|^2~\dO,\quad \quad \norm[\Sigma]{\bfsigma}^2:=\int_{\O}|\bfsigma|^2~\dO+\int_{\O}|\bdiv \bfsigma|^2~\dO.
\end{equation*}
From now on, for sake of simplicity, we only consider the case of homogeneous boundary conditions, aware that the general case can be treated exactly in the same way of the classical Galerkin methods. Hence we take $\bbg = \bfpsi=\bfzero$.
We define the bilinear forms $a(\cdot,\cdot): \Sigma \times \Sigma \rightarrow \R$ and $b(\cdot,\cdot): \Sigma \times U \rightarrow \R$ as follows
\begin{equation}
    \begin{aligned}
        & a(\bfsigma,\bftau):=\int_{\O}\D\, \bfsigma:\bftau~\dO,\\ 
        & b(\bfsigma,\bbu):=\int_{\O}\bdiv\bfsigma\cdot \bbu~\dO,
    \end{aligned}
\end{equation}
where the tensor $\D=\C^{-1}$ is assumed to be uniformly bounded, positive definite and sufficiently regular. Then, the corresponding weak formulation reads
\begin{equation}\label{problem:cont-weak}
\left\lbrace{
	\begin{aligned}
	&\mbox{find } (\bfsigma,\bbu)\in \Sigma\times U~\mbox{such that}\\
	&a(\bfsigma,\bftau) + b(\bftau, \bbu)=\bfzero  &\forall \bftau\in \Sigma,\\
	& b(\bfsigma, \bbv) = -(\bbf,\bbv) &\forall \bbv\in U,
	\end{aligned}
} \right.
\end{equation}
where $(\cdot,\cdot)$ is the inner product in  $\left[L^2(\O)\right]^3$. It is well known that Problem~\eqref{problem:cont-weak} is well-posed, see for instance~\cite{BoffiBrezziFortin} and it holds
$$
\norm[\Sigma]{\bfsigma}+\norm[U]{\bbu}\lesssim \norm[0]{\bbf}, 
$$
where the hidden constant depends on the domain $\O$ and on the material tensor $\D$, which does not degenerate in the incompressible limit. 
\section{The Virtual Element Method}\label{section:VEMHRFAMILY}
In this section, we define our virtual element discretization of Problem~\eqref{problem:cont-weak}.
Let $\{\mathcal{T}_h\}_h$ be a sequence of decompositions of $\Omega$ into general polyhedral elements $E$ with
\[
%
h := \max_{E \in \mathcal{T}_h} h_E.
\]
We suppose that for all $h$, each element $E$ in $\mathcal{T}_h$ is a contractible polyhedron that fulfils the following assumptions, see~\cite{projectors}:
\begin{enumerate}[label=\textbf{A.\arabic*}]
	\item\label{meshA_1} $E$ is star-shaped with respect to a ball $B_E$ having radius $\geq \gamma \, h_E$;
	\item\label{meshA_2} every face $f$ of $E$ is star-shaped with respect to a disk $B_f$ having radius $\geq \gamma\, h_f$;
	\item\label{meshA_3} every edge $e$ of $E$ satisfies $h_e \geq \gamma \, h_f \geq \gamma^2 \, h_E$,
\end{enumerate}
where $\gamma$ is a suitable positive constant. We remark that the above hypotheses, although not too restrictive in many practical cases, can be further relaxed, as investigated in~\cite{BLRXX, BERTOLUZZA2021,BrennerSungSmallEdges}. 
\subsection{The local spaces}\label{subsection:localSpaces}
We fix an integer $k\geq 1$. Given a polyhedron $E\in\Th$, with $n^E_f$ faces, we firstly introduce these two elementary spaces: $\RM(E)$ and $\RM_k^{\perp}(E)$.
\paragraph{Space $\RM(E)$} It is the space of local infinitesimal rigid body motions:
\begin{equation}\label{space:localRM}
	\RM(E):=
	\left\{
	\bbr(\bbx) = \bfalpha + \bfomega \wedge \big(\bbx -\bbx_E\big)\ \ \text{ s.t. }\ \bfalpha,\,\bfomega \in\R^3 
	\right\},
\end{equation}
whose dimension is equal to 6.
\paragraph{Space $\RM_k^{\perp}(E)$} This space represents the orthogonal space of $\RM(E)$ respect to $\PolyTriple{k}{E}$ and it is defined as follows:
\begin{equation}\label{space:localRM_perp}
\RM_k^{\perp}(E):=
\left\{
\bbp_k \in\PolyTriple{k}{E} \ : \ \int_{E} \bbp_k\cdot \bbr~\dEl =0, \quad \forall \bbr \in \RM(E)
\right\}.
\end{equation}
Hence, the following $L^2$-orthogonal decomposition holds:
\begin{equation}\label{space:decomposition_orth_poly}
\PolyTriple{k}{E}  := \RM(E) \oplus \RM_k^{\perp}(E).
\end{equation}
The dimension of the space $\RM^{\perp}_k(E)$ is 
\begin{equation*}
\pi^{\perp}_{k,E}:= \dim\left(\PolyTriple{k}{E} \right) - \dim\left(\RM(E)\right) = 3\pi_{k,E} - 6= 
\frac{k^3+6k^2+11k-6}{2}.
\end{equation*}
A possible construction of a basis $\left\{\bs\varphi_i \right\}_{i=1, \dots,\pi_{k,E}^{\perp}}$ for $\RM_k^{\perp}(E)$ can be obtained as follows. Let $\left\{\hat{\bs\varphi}_i \right\}_{i=1,\dots,3\pi_{k,E}}$ be a set of linear independent basis function for $\PolyTriple{k}{E}$ such that the first six functions are rigid body motions. Then, starting from this set, we perform an $L^2$-orthogonalization procedure. In particular, the approach that we have used in Section~\ref{section:numerical_results} is based on the application of the modified Gram-Schmidt (MGS) orthogonalization algorithm~\cite{BassiBottiColomboDiPietroTesini, GiraudLangouRozioznik}. The MSG algorithm with re-orthogonalization is set up in Algorithm~\ref{alg:1}.
\begin{algorithm}[tbh]
\caption{MGS algorithm}\label{alg:1}
\begin{algorithmic}[1]
\For{$i=1:3\pi_{k,E}$}
	\For{$k=1:2$} \hskip3em\Comment{re-orthoganlization procedure}	
	\State $\hat{\bs\varphi}_i^{(k)} = \hat{\bs\varphi}_i^{(k-1)}$
	\For {$j=1:i-1$} \Comment{GS orthogonalization}
	\State$\hat{\bs\varphi}_i^{(k)} = \hat{\bs\varphi}_i^{(k)}-(\hat{\bs\varphi}_i^{(k)},\bs\varphi_j)\,\bs\varphi_j$
	\EndFor%
\EndFor
	\State		$\bs\varphi_i=\hat{\bs\varphi}_i^{(k)}/||\hat{\bs\varphi}_i^{(k)}||_{0,E}$
	\hskip2em\Comment{Normalization}
	\EndFor
\end{algorithmic}
\end{algorithm}
As we can see, the term ``re-orthogonalization'' is related to the fact that the orthogonalization procedure is applied more than once, in order to obtain a more stable algorithm. In particular, for our code, we applied twice, which is enough to achieve the goal.
%
\paragraph{Stress space} 
Now, we are ready to introduce our local approximation space for the stress field.
\begin{equation}\label{space:localStress}
\begin{aligned}
\Sigma_h(E):=
\left\{\bftau_h\in  H(\bdiv;E)\ :\  \right.&\exists\, \bbw^\ast\in \left[H^1(E)\right]^3 \mbox{ s.t. } \bftau_h=\C\teps(\bbw^\ast);\\
& \left. (\bftau_h\bbn)_{|f}\in \PolyTriple{k}{f} \quad \forall f\in \partial E; \right.\\
& \left.\bdiv\bftau_h\in\PolyTriple{k}{E}
\right\}.
\end{aligned}
\end{equation}
We notice that the stress approximation space consists of regular symmetric tensors that they are unknown (virtual) inside the element, while their tractions on each faces and the divergence are vector polynomial functions. Moreover, due to decomposition~\eqref{space:decomposition_orth_poly}, for each $\bftau_h\in\Sigma_h(E)$ we may write is divergence as follows
%
\begin{equation}
\bdiv \bftau_h := \bbd_{RM} + \bbd_k^{\perp}
\end{equation}
for a unique couple $(\bbd_{RM},\bbd_k^{\perp})\in\RM(E)\times\RM_k^{\perp}(E)$. 
We observe that $\bbd_{RM}$, the $\RM(E)$-component of $\bdiv \bftau_h$, is completely determine by the boundary information $(\bftau_h\bbn)_{|f}\in\PolyTriple{k}{f}$ (cf.~\eqref{space:localStress}). Indeed, using the integration by parts and the orthogonal decomposition~\eqref{space:decomposition_orth_poly}, we have:
\begin{equation}~\label{eq:divergenceRM}
	\int_E \bbd_{RM}\cdot \bbr~\dEl = \int_E \bdiv\bftau_h\cdot \bbr~\dEl= \sum_{f\in \partial E} \int_{f}\left(\bftau_h\bbn\right)_{f}\cdot \bbr~\df \qquad \forall\bbr\in\RM(E).
\end{equation} 
More precisely, setting
\begin{equation}~\label{eq:div0}
	\bbd_{RM} = \bfalpha_E + \bfomega_E \wedge \big(\bbx -\bbx_E\big),
\end{equation}
from~\eqref{eq:divergenceRM} and Proposition~3.1 in~\cite{DLV}, we infer
\begin{equation}~\label{eq:div1}
	\bfalpha_E = \frac{1}{|E|}\sum_{f\in \partial E} \int_f \left(\bftau_h\bbn\right)_{f}~\df,
\end{equation}
and $\bfomega_E$ is the unique solution of the following linear system
\begin{equation}~\label{eq:div2}
	\int_E (\bbx -\bbx_E) \wedge \left[\bfomega_E\wedge(\bbx -\bbx_E)\right]~\dEl = \sum_{f\in \partial E}\int_f (\bbx-\bbx_E)\wedge \left(\bftau_h\bbn\right)_{f}~\df.
\end{equation}
Accordingly, for the local space $\Sigma_h(E)$ we can choose the following degrees of freedom:
\begin{itemize}
\item for each face $f$ of $\partial E$, we take the $3\pi_{k,f}$ boundary moments
\begin{equation}~\label{eq:stress_boundaryDofs}
\bftau_h \rightarrow \int_f \bftau_h\bbn\cdot \bbp_k~\df\qquad \forall \bbp_k \in \PolyTriple{k}{f}; 
\end{equation}
\item for each element $E$, we consider the $\pi_{k,E}^{\perp}$ moments of the divergence
\begin{equation}~\label{eq:stress_divergenceDofs}
\bftau_h \rightarrow  \int_E \bdiv\bftau_h \cdot \bbr_k^{\perp}~\dEl \qquad \forall \bbr_k^{\perp} \in \RM_k^{\perp}(E). 
\end{equation}
\end{itemize}
So, we infer that the dimension of the space~\eqref{space:localStress} is
\begin{equation*}
\dim\left(\Sigma_h(E)\right) := 3 n^E_f\,\pi_{k,f}  + \pi_{k,E}^{\perp}.
\end{equation*}
\begin{remark}
The proof that the above linear operators constitute a set of degrees of freedom for $\Sigma_h(E)$ follows the same steps of Lemma 3.1. in~\cite{ARTIOLI2018978}.
\end{remark}
\paragraph{Displacement space} The local approximation space for the displacement field is simply defined by
\begin{equation}~\label{space:localDisplacement}
	U_h(E)=\left\{ \bbv_h \in \LTwoTriple{E} \ : \ \bbv_h \in\PolyTriple{k}{E}\right\}.
\end{equation}
Accordingly, for the local space $U_h(E)$ we can take the following degrees of freedom:
\begin{equation}~\label{eq:displacement_dofs}
	\bbv_h \rightarrow \int_E \bbv_h \cdot \bbp_k~\dEl, \qquad \forall \bbp_k \in \PolyTriple{k}{E}.
\end{equation}
It follows that the dimension of this space is $$\dim(U_h(E)):=3\pi_{k,E}.$$ 

\begin{figure}[tbh]
	\centering
	\includegraphics[width=0.3\linewidth]{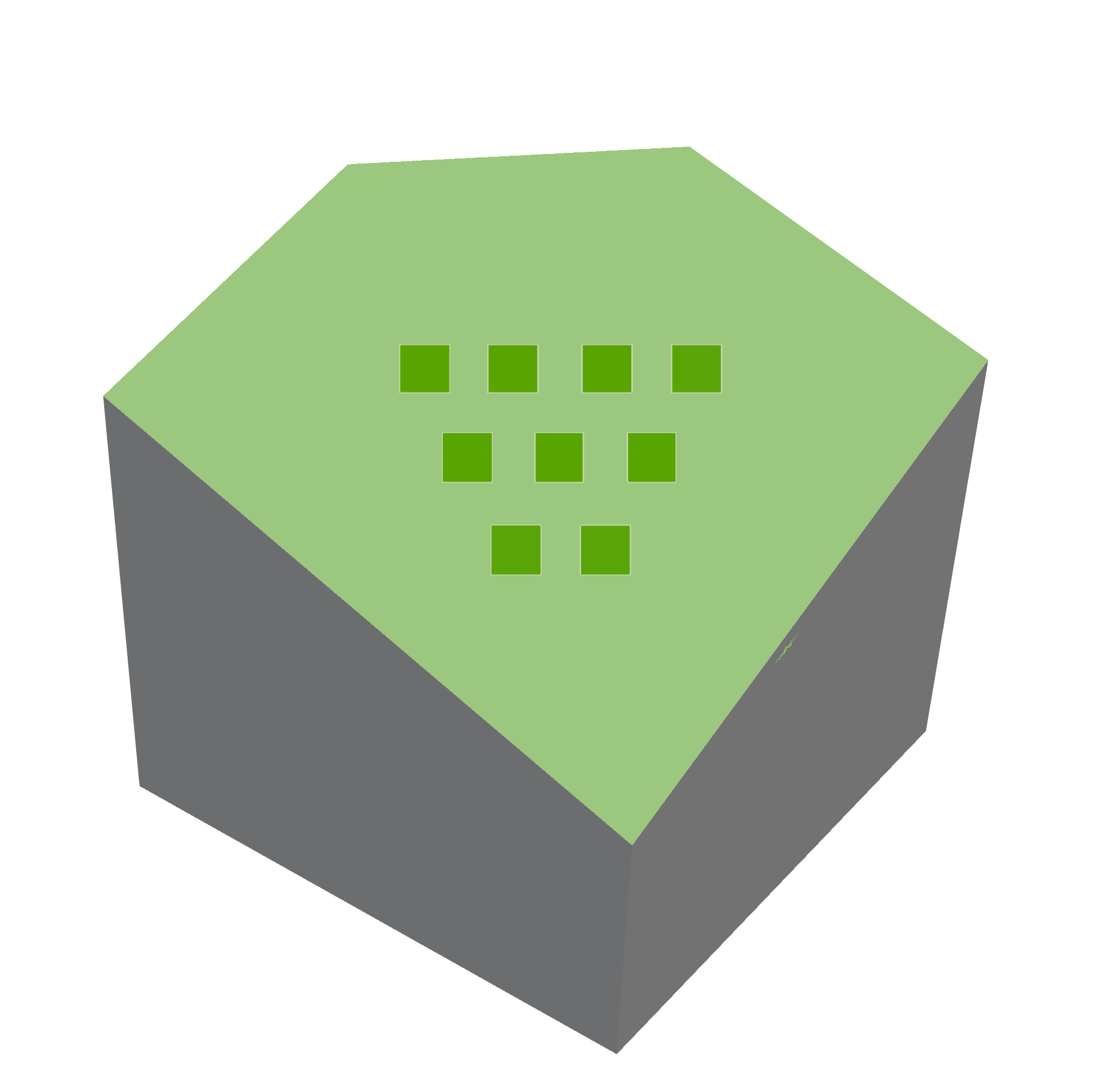}
	\quad
	\includegraphics[width=0.3\linewidth]{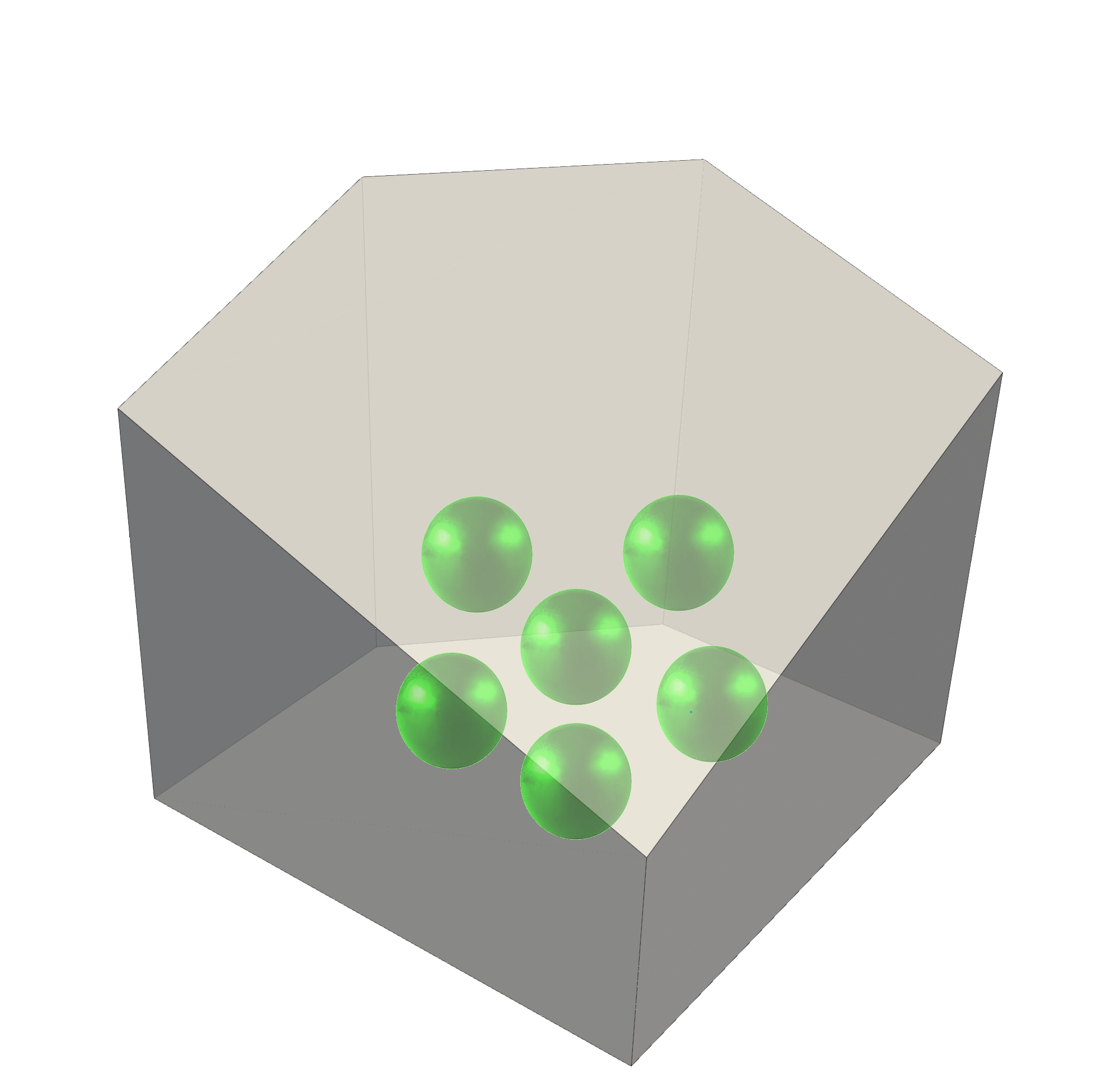}
	\quad
	\includegraphics[width=0.3\linewidth]{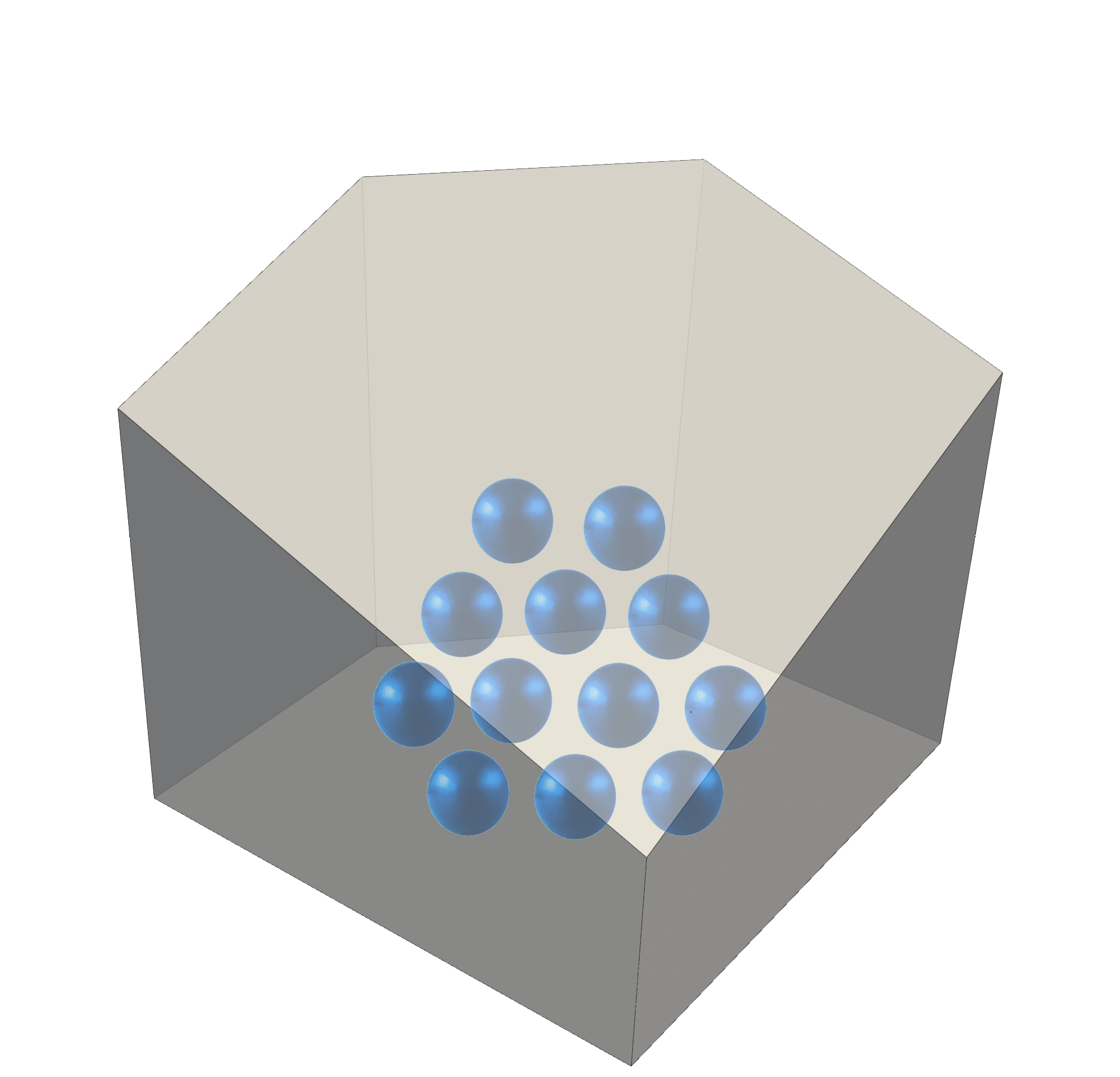}
	\caption{Overview of the local degrees of freedom for k=1.}
	\label{fig:dofs}
\end{figure}

In Fig.~\ref{fig:dofs} we schematically depict the local degrees of freedom for both the stress and the displacement field, in the case $k=1$. More precisely we have that: on the left the dark green squares represent the boundary moments (cf.~\eqref{eq:stress_boundaryDofs}) on a fixed face $f$; in the middle, the green spheres represent the moments of the divergence (cf.~\eqref{eq:stress_divergenceDofs}); on the right, the blue spheres represent the displacement degrees of freedom (cf.~\eqref{eq:displacement_dofs}).
\subsection{The local bilinear forms}~\label{ss:localBilinearForms}
We introduce the VEM counterparts of the local forms associated with the continuous problem.
\paragraph{The local mixed term $b_E(\cdot,\cdot)$} For every $\bftau_h\in\Sigma_h(E)$ and $\bbv_h\in U_h(E)$, the term
\begin{equation}~\label{eq:computabilityMixedTerm}
	b_E(\bftau_h,\bbv_h):=\int_E \bdiv\bftau_h \cdot \bbv_h~\dEl 
\end{equation} 
is computable via degrees of freedom. As a consequence, we do not need to introduce any approximation of the terms $(\bdiv \bftau, \bbu)$ and $(\bdiv \bfsigma, \bbv)$ in Problem~\eqref{problem:cont-weak}.
\paragraph{The local bilinear form $a_E(\cdot,\cdot)$} The local bilinear form
\begin{equation}\label{eq:discreteLocalBilinearForm}
	a_E(\bfsigma_h,\bftau_h)=\int_E\D \bfsigma_h \colon \bftau_h~\dEl
\end{equation}
is not computable for a general couple $(\bfsigma_h, \bftau_h)\in\Sigma_h(E)\times\Sigma_h(E)$.
Then, follows the standard VEM approach (see~\cite{volley}, for instance) in order to build a computable approximation of the bilinear form $a_E(\cdot,\cdot)$, we need to define a suitable projection operator onto local polynomial functions.

For each element $E\in\Th$, we firstly introduce the local space
\begin{equation}~\label{eq:localHatSigma}
	\hat{\Sigma}(E):= \left\{\bftau \in H(\bdiv,E) \, : \,  \exists\bbw\in \left[H^1(E)\right]^3 \mbox{ s.t. } \bftau=\C\teps(\bbw)\right\}
\end{equation}
and the global space
\begin{equation}~\label{eq:globalHatSigma}
\hat{\Sigma}:= \left\{\bftau \in H(\bdiv,\O) \, : \,  \exists\bbw\in \left[H^1(\O)\right]^3 \mbox{ s.t. } \bftau=\C\teps(\bbw)\right\}.
\end{equation}
We define the local projection operator $$\Pi_E^k: \hat{\Sigma}(E)\rightarrow T_k(E)$$
by requiring
\begin{equation}~\label{eq:projectionDefinition}
a_E(\Pi_E^k\bftau, \bs\pi_k) = a_E(\bftau,\bs \pi_k) \quad \forall \bs\pi_k \in T_k(E),\end{equation}
with 
\begin{equation}~\label{eq:spaceTraction} 
T_k(E):=
\left\{\C \teps(\bbp_{k+1}) \, : \, \bbp_{k+1} \in\PolyTriple{k+1}{E}\right\}.
\end{equation}
\begin{remark}
Alternatively to the condition~\eqref{eq:projectionDefinition}, one may find  $\bbp_{k+1}\in\PolyTriple{k+1}{E}$ such that 
\begin{equation}~\label{eq:projectionDefinition1}
\int_{E} \C \teps (\bbp_{k+1}) \, : \, \teps(\bbq_{k+1})~\dEl = \int_{E} \bftau \, : \, \teps(\bbq_{k+1})~\dEl, \qquad \forall \bbq_{k+1}\in\PolyTriple{k+1}{E}
\end{equation}
where the polynomial functions $\bbp_{k+1}$ are defined up to a rigid body motion.
\end{remark}
Then, the approximation of $a_E(\cdot,\cdot)$ reads
\begin{equation}\label{eq:localDiscreteBilinearForm}
\begin{aligned}
	a_E^h(\bfsigma_h,\bftau_h) &= a_E(\Pi_E^k \bfsigma_h, \Pi_E^k \bftau_h) + s_E((I -\Pi_E^k)\bfsigma_h,(I -\Pi_E^k)\bftau_h) \\
	&= \int_E \D \Pi_E^k\bfsigma_h \, :\, \Pi_E^k\bftau_h~\dEl + s_E((I -\Pi_E^k)\bfsigma_h,(I -\Pi_E^k)\bftau_h) \\
\end{aligned}
\end{equation}
where $s_E(\cdot,\cdot)$ is a suitable stabilization term. We propose
\begin{equation}\label{eq:stabilization}
	s_E(\bfsigma_h,\bftau_h)= \kappa_E h_E \int_{\partial E} \bfsigma_h \bbn \cdot \bftau_h \bbn~\df
\end{equation}
where $\kappa_E$ is a positive constant. For instance, in the numerical examples of Section~\ref{section:numerical_results} we take $\kappa_E=\frac{1}{2}tr(\D_{|E})$).
\begin{remark}
    A possible variant of~\eqref{eq:stabilization} is supplied by
    \begin{equation}
    s_E(\bfsigma_h,\bftau_h)= \kappa_E |E| \sum_{f\in\partial E} \frac{1}{h_f}\int_{f} \bfsigma_h \bbn \cdot \bftau_h \bbn~\df,
    \end{equation}
    which could better mimic the shape of the element.
\end{remark}

\paragraph{The local loading term}
The loading term, see~\eqref{problem:cont-weak}, is simply:
\begin{equation}~\label{eq:loadingTerm}
	(\bbf,\bbv_h):= \int_{\Omega} \bbf\cdot \bbv_h~\dO = \sum_{E\in\Th} \int_{E}\bbf\cdot\bbv_h~\dEl.
\end{equation}
Since $\bbv_h\in\PolyTriple{k+1}{E}$, the right-hand side is computable via a suitable quadrature rule for polyhedral domains, see, i.e.,~\cite{Sommariva2009}.
\subsection{The discrete scheme}
Starting from the local spaces and the local terms introduced in the previous subsections, we can set the global problem.
First of all, we introduce the global approximation space for the stress field, by glueing the local approximation spaces, see~\eqref{space:localStress}:
\begin{equation}\label{space:globalStress}
\Sigma_h :=\left\{\bftau_h \in H(\bdiv; \Omega) \, : \, \bftau_{h_{|E}}\in\Sigma_h(E) \quad \forall E \in \Th\right\}.
\end{equation}
As it is well known, since $\Sigma_h\subseteq H(\bdiv;\Omega)$, we require the continuity of the stress boundary degrees of freedom on each internal interface of the mesh $\Th$. Therefore, give an internal face, we establish once and for all its normal vector and then we uniquely define the corresponding functional~\eqref{eq:stress_boundaryDofs}. Instead, for the functional~\eqref{eq:stress_divergenceDofs} no inter-element continuity is required.
For the global approximation of the displacement field, we take, see~\eqref{space:localDisplacement}:
\begin{equation}\label{space:globalDisplacement}
 U_h=\left\{\bbv_h \in \LTwoTriple{\Omega} \, : \, \bbv_{h_{|E}} \in U_h(E) \quad \forall E\in\Th  \right\}.
\end{equation}
Finally, given a local approximation of $a_E(\cdot,\cdot)$, see~\eqref{eq:localDiscreteBilinearForm}, we set 
\begin{equation}
	a_h(\bfsigma_h,\bftau_h) := \sum_{E\in\Th} a_E^{h}(\bfsigma_h,\bftau_h).
\end{equation}
The method we consider is then defined by
\begin{equation}\label{problem:discrete}
\left\lbrace{
	\begin{aligned}
		&\mbox{find } (\bfsigma_h,\bbu_h)\in\Sigma_h\times U_h~\mbox{such that} &\\
		&a_h(\bfsigma_h,\bftau_h) + b(\bftau_h,\bbu_h)  =\bfzero &  \forall \bftau_h \in\Sigma_h,\\
		& b(\bfsigma_h,\bbv_h) = -(\bbf,\bbv_h) & \forall \bbv_h \in U_h.
	\end{aligned}
} \right.
\end{equation}
\section{Stability and convergence analysis}~\label{section:StabilityConvergenceAnalysis}
Since some results of the analysis follows the guidelines of the theory developed in~\cite{ARTIOLI2017155,ARTIOLI2018978} for 2D problems, in this section we do not provide full details of the proofs. From now on, for sake of simplicity, we will consider the problem only with homogeneous natural boundary conditions. Firstly, we introduce this useful regular space.
%

%
Given a measurable subset $D\subseteq \Omega$ and $r>2$, we define 
\begin{equation}
	W^r(D):=\left\{\bftau \ : \ \bftau \in \LpTensor[r]{D}, \quad \bdiv\bftau \in\LTwoTriple{D} \right\},
\end{equation}
equipped with the obvious norm.
\subsection{Interpolation operators for stresses}
We now introduce a local interpolation operator $\I_E: W^r(E)\rightarrow \Sigma_h(E)$.
Given $\bftau\in W^r(E)$, we define its interpolant $\I_E\bftau\in\Sigma_h(E)$ such that
\begin{equation}\label{eq:systemInterpolationOperator}
\left\{
    \begin{aligned}
        &\int_{\partial E} \left(\I_E \bftau \right)\bbn \cdot \bs\varphi_k~\df =\int_{\partial E}\bftau \bbn \cdot \bs\varphi_k~\df  \qquad &\forall\bs\varphi_k\in R_k(\partial E), \\ 
        &\int_{E} \bdiv(\I_E \bftau ) \cdot \bs\psi_k~\dEl=  \int_{E} \bdiv(\I_E \bftau ) \cdot \bs\psi_k~\dEl \qquad &\forall \bs\psi_k \in \RM_k^{\perp}(E), 
    \end{aligned}
\right.
\end{equation}
where the space $R_k(\partial E)$ is defined by
\begin{equation}
	R_k(\partial E)= \left\{\bs\varphi_k\in\LTwoTriple{\partial E} \, : \, \bs\varphi_{k|f}\in\PolyTriple{k}{f}, \quad \forall f \in \partial E \right\}.
\end{equation}
\begin{remark}
If $\bftau$ is a regular function, the first condition of~\eqref{eq:systemInterpolationOperator} is equivalent to require
\begin{equation}
	\int_{f} \left(\I_E \bftau\right) \bbn\cdot \bbq_k~\df =\int_f \bftau \bbn\cdot \bbq_k~\df \qquad \forall \bbq_k\in\PolyTriple{k}{f};
\end{equation}
otherwise, the integral of the right-hand side in~\eqref{eq:systemInterpolationOperator} must be interpreted as a duality between $\left[W^{-\frac{1}{r},r}(\partial E)\right]^3$ and $\left[W^{\frac{1}{r},r'}(\partial E)\right]^3$. 
\end{remark}
%
Due to the unisolvence of degrees of freedom,  the local interpolant $\I_E\bftau$ is well-defined by the conditions in~\eqref{eq:systemInterpolationOperator}. The global interpolation operator $\I_h: W^r(\O)\rightarrow \Sigma_h$ is simply defined by gluing the local contributions $\I_E$ as follows
\begin{equation}
\left( \I_h\bftau \right)_{|E} := \I_E\bftau, \quad \forall E\in\Th, \,\, \forall \bftau\in W^r(\O).
\end{equation}
Moreover, due to its definition, the commuting diagram property holds

\begin{figure}[hbt]
    \centering
      \begin{tikzpicture}[->,>=stealth',shorten >=1pt,auto,node distance=2cm,thick, main node/.style={font=\sffamily\large\bfseries}]
  \node[main node] (1) {$\Sigma$};
  \node[main node] (2) [right of=1] {$U$};
  \node[main node] (3) [below of=1] {$\Sigma_h$};
  \node[main node] (4) [below of=2] {$U_h$};
  \node[main node] (5) [right of=2] {$\bfzero$};
  \node[main node] (6) [right of=4] {$\bfzero$};
  \path[every node/.style={font=\sffamily\small}]
    (1) edge node [above] {$\bdiv$} (2)
        edge node [left] {$\I_h$} (3)
    (2) edge node [right] {$\P_h^k$} (4)
	    edge node [above] {} (5)
    (3) edge node [below] {$\bdiv$} (4)
    (4) edge node [below] {} (6);
\end{tikzpicture}
\end{figure}\label{fig:diagram}

so that 
\begin{equation}~\label{eq:commuting_diagram_property}
		\bdiv(\I_h\bftau) = \P_h^k(\bdiv\bftau) \qquad \forall \bftau \in W^r(\O)
\end{equation}
where $\P_h^k:U\rightarrow U_h$ denotes the $L^2$-projection operator onto the piecewise polynomial functions of degree up to $k$.
\subsection{The ellipticity-on-the-kernel and the inf-sup condition}
By definition of the discrete spaces~\eqref{space:localStress}, \eqref{space:globalStress} and~\eqref{space:localDisplacement}, \eqref{space:globalDisplacement}, we notice that:
\begin{equation}
	\bdiv(\Sigma_h) \subseteq U_h.
\end{equation}
As a consequence, introducing the discrete kernel $K_h\subseteq\Sigma_h$:
\begin{equation}
	K_h=\left\{ \bftau_h \in\Sigma_h \ : \ (\bdiv \bftau_h, \bbv_h)=0 \quad \forall v_h\in U_h \right\},
\end{equation}
we infer that $K_h\subseteq K$, where 
\begin{equation}
    K=\left\{\bftau \in \Sigma \ : \ \left(\bdiv \bftau, \bbv \right) \quad \forall \bbv \in U \right\}.
\end{equation}
Hence, it holds:
\begin{equation}
	\norm[\Sigma]{\bftau_h} = \norm[0]{\bftau_h}\qquad \forall \bftau_h\in K_h.
\end{equation}
This is essentially the property which leads to the ellipticity on the kernel condition:
\begin{prop}
Fixed $k\geq 1$, for the proposed method, there exists a constant $\alpha_*>0$ such that
\begin{equation}
	a_h(\bftau_h, \bftau_h) \geq \alpha_* \norm[\Sigma]{\bftau_h}^2 \qquad \bftau_h \in K_h.
\end{equation}
\end{prop}
Moreover, as a consequence of the commuting diagram property, see~\eqref{eq:commuting_diagram_property}, and the theory developed in~\cite{ARTIOLI2017155,BoffiBrezziFortin}, the following discrete inf-sup condition holds.
\begin{prop} Fix the integer $k\geq 1$. Suppose that the mesh assumptions~\ref{meshA_1}, \ref{meshA_2} and \ref{meshA_3} are fulfilled. There exists $\beta > 0$, independent of $h$, such that 
\begin{equation}
\inf_{\bbv_h\in U_h} \sup_{\bftau_h\in\Sigma_h} \frac{b(\bftau_h,\bbv_h)}{\norm[U]{\bbv_h}\norm[\Sigma]{\bftau_h}}\geq \beta.    
\end{equation}
\end{prop}
\subsection{Local approximation estimates}
For the local projection operator $\Pi_E^k$, see~\eqref{eq:projectionDefinition}, using similar steps detailed in~\cite{ARTIOLI2018978}, one can prove the following result.
\begin{prop}
	Fixed $k\geq 1$ and let $r$ be such that $0\leq r\leq k+1$. Under assumptions \ref{meshA_1}, \ref{meshA_2} and \ref{meshA_3}, for the projection operator $\Pi_E^k$ defined in~\eqref{eq:projectionDefinition}, the following estimate holds:
	\begin{equation}
		\norm[0,E]{\bftau -\Pi_E^k\bftau}\lesssim h_E^r \seminorm[{r},E]{\bftau} \qquad \forall \bftau\in \hat{\Sigma}_h(E)\cap \left[H^r(E)\right]^{3\times 3}_s.
	\end{equation}
\end{prop}
Instead, for the local interpolation operator $\I_E$ (cf.~\eqref{eq:systemInterpolationOperator}) we have the following result.
\begin{prop}
	Fix $k\geq 1$, let $r$ be such that $1\leq r \leq k+1$. Under assumptions~\ref{meshA_1}, \ref{meshA_2} and \ref{meshA_3}. Assuming $\bftau\in\hat{\Sigma}(E)\cap \left[ H^r(E)\right]^{3 \times 3}_s$ and $\bdiv\bftau \in \left[ H^r(E)\right]^3$, the following estimates hold true:
	\begin{equation}\label{eq:interpolationEstimate1}
		\norm[0,E]{\bftau -\I_E \bftau} \lesssim h_E^r |\bftau|_{r,E} 
	\end{equation}
	and
	\begin{equation}\label{eq:interpolationEstimate2}
		\norm[0,E]{\bdiv(\bftau -\I_E \bftau)}\lesssim h_E^r |\bdiv \bftau|_{r,E}.
	\end{equation}
\end{prop}
\begin{proof}
	Let's start proving~\eqref{eq:interpolationEstimate1}. Since $\bftau\in\hat{\Sigma}(E)\cap \left[ H^r(E)\right]^{3 \times 3}_s$, there exists a $\bbw\in\left[H^1(E) \right]^{3}$ such that $\bftau = \C \teps(\bbw)$. Similarly, for its interpolant  $\I_E\bftau\in\Sigma_h(E)$, there exists another function $\bbw^*\in\left[H^1(E) \right]^{3}$ such that $\I_E\bftau=\C\teps(\bbw^*)$.
	Now, setting $\bs \xi =\left(\bbw-\bbw^*\right)\in \left[H^1(E)\right]^3$ we have that 
	$\bftau -\I_E\bftau= \C\teps(\bs\xi)$.
	Then, using the definition of the interpolation operator, together with~\eqref{eq:div0}, \eqref{eq:div1} and \eqref{eq:div2} we infer that $\bs\xi\in \left[H^1(E)\right]^3$ can be seen as solution of the following pure traction problem:
	\begin{equation}~\label{eq:problemPureTraction1}
		\left\{
		\begin{aligned}
			&-\bdiv\left(\C\teps(\bs\xi)\right) = \bbg \quad &\mbox{in}\ \Omega,\\
			& \left(\C\teps(\bs\xi)\right) \bbn = \bbh \quad &\mbox{on}\ \partial \Omega, 
		\end{aligned}
		\right.
	\end{equation}
	where 
	\begin{equation}~\label{eq:problemPureTraction1_rhs}
		\left\{
		\begin{aligned}
			\bbg &= \bdiv \left(\I_E\bftau\right) -\bdiv \bftau =\P_E^k\left(\bdiv\bftau\right)-\bdiv \bftau,
			\\
			\bbh &= \sum_{f\in\partial E} \left( \bftau \bbn  -\left(\I_E \bftau\right) \bbn \right)\chi_f =\bftau \bbn  -\P_{\partial E}^k\left(\bftau \bbn\right),
		\end{aligned}
		\right.
	\end{equation}
	and $\chi_f$ denotes the characteristic function of the face $f$, whereas 
	$\P_E^k$ and $\P_{\partial E}^k$ indicate the $L^2$-projection operators onto the polynomial functions on $E$ and onto the piecewise polynomial functions on $\partial E$ (with respect to the face $f\in\partial E$).
	Applying~\cite[Lemma 5.1]{ARTIOLI2017155}, we get
	\begin{equation}~\label{eq:estimateLemma5.1}
		\norm[0,E]{\C\left( \teps (\bs\xi)\right)}\lesssim h_E \norm[0,E]{\bbg} + h_E^{1/2} \norm[0,\partial E]{\bbh}.
	\end{equation}
	Now, we need to estimate $\bbg$ end $\bbh$. 
	Therefore, let $r$ be such that $1\leq r \leq k+1$ and using standard approximation estimates, from the first equation of~\eqref{eq:problemPureTraction1_rhs}, we have
	\begin{equation}~\label{eq:estimate_g}
		\norm[0,E]{\bbg} = \norm[0,E]{\P_E^k\bdiv \bftau - \bdiv \bftau}\lesssim \norm[0,E]{\bdiv \bftau} \lesssim h^r\seminorm[r,E]{\bdiv\bftau}.
	\end{equation}
	For the second equation of~\eqref{eq:problemPureTraction1_rhs}, always taking $1\leq r \leq k+1$ and using standard approximation estimates and trace inequality, we get
	\begin{equation}~\label{eq:estimate_h}
		\begin{aligned}
			\norm[0,\partial E]{\bbh} = \norm[0,\partial E]{\bftau\bbn -\P_{\partial E}^k(\bftau \bbn)} &\lesssim \norm[0,\partial E]{\bftau \bbn} \lesssim \norm[0,\partial E]{\bftau} \\
			&\lesssim h_E^{1/2}\seminorm[1/2,\partial E]{\bftau} \lesssim h_E^{1/2}\seminorm[1,E]{\bftau}\\ &\lesssim h_E^{r+1/2}\seminorm[r,E]{\bftau}.
		\end{aligned}
	\end{equation}
	Taking into account~\eqref{eq:estimate_g} and~\eqref{eq:estimate_h}, from~\eqref{eq:estimateLemma5.1} we obtain estimate~\eqref{eq:interpolationEstimate1}. 
	The estimate~\eqref{eq:interpolationEstimate2} immediately follows from~\eqref{eq:estimate_g}:
	\begin{equation}
		\norm[0,E]{\bdiv\left(\bftau- \I_E\bftau\right)}=\norm[0,E]{\bbg} \lesssim h_E^{r}\seminorm[r,E]{\bdiv \bftau},
	\end{equation}
	concluding the proof.
\end{proof}
\subsection{Error estimates}
Using the same techniques developed in~\cite{ARTIOLI2018978,DLV2021}, one can prove the following result.
\begin{thm}~\label{theorem:convergenceAndSuperconvergence}
	Let $k$ be an integer with $k\geq 1$. Let $\left(\bfsigma,\bbu\right)\in\Sigma\times U$ be the solution of the continuous Problem~\eqref{problem:cont-weak}, and $\left(\bfsigma_h,\bbu_h\right)\in\Sigma_h\times U_h$ be the discrete stress and displacement solution of the discrete Problem~\eqref{problem:discrete}. Under the mesh assumptions $\mathbf{A1}$, $\mathbf{A2}$ and $\mathbf{A3}$ and supposing $\left(\bfsigma,\bbu\right)$ sufficiently regular, the following estimates hold true 
	\begin{equation}~\label{eq:error_estimates}
		\norm[\Sigma]{\bfsigma-\bfsigma_h} +\norm[U]{\bbu -\bbu_h} \lesssim h^{k+1}
	\end{equation}
	and
	\begin{equation}~\label{eq:superconvergence}
		\norm[U]{\P_h^{k}\bbu -\bbu_h} \lesssim h^{k+2}
	\end{equation}
	where we recall that $\P^{k}_h$ is the $L^2$-projection of $\bbu$ onto the piecewise polynomial function of degree up to $k$.
\end{thm}

	\section{Hybridization technique and post-processing procedure}~\label{section:hybridization_postprocessing}
In this section, we briefly present the main idea of the hybridization procedure and the advantages that this technique leads, as the possibility to reconstruct a better discrete solution for the displacement field, see~\cite{ArnoldBrezzi_1985,DLV2021} for more details.
\subsection{Hybridization technique}
The hybridization technique is a computational procedure used to solve mixed PDE problems in order to obtain some theoretical and practical benefits~\cite{ArnoldBrezzi_1985}.
Essentially, this technique applies whenever the discrete space for the stress field does not have nodal degrees of freedom, namely when the continuity constraints are imposed on the interfaces on the elements, as for our methods. Once this necessary condition is satisfied, the hybridization consists of the following two phases: the first step is characterized by the imposition of the $H(\bdiv)$-conformity through the Lagrange multipliers, while the second one by the application of the static condensation algorithm to obtain a (smaller) symmetric and positive linear system, instead of the original indefinite one. 
Therefore, we firstly introduce the following discerte space
\begin{equation}
	\tilde{\Sigma}_h(\Th) := \left\{ \bftau_h \in \LpTensor[2]{\Omega} \ \colon \ {\bftau_h}_{|E} \in\Sigma_h(E), \ \forall E\in\Th \right\}
\end{equation}
which is a subspace of $\Sigma_h = \tilde{\Sigma}_h(\Th) \cap H(\bdiv,\O)$. 
Given $\Fh^I$, the set of the internal faces, we define the space of the Lagrange multipliers, see~\eqref{space:localStress}
\begin{equation}
	\Lambda_h(\Fh^I):=\left\{ \bs\nu_h \in\LTwoTriple{\Fh^I}\,: \, \bs\nu_h\in\PolyTriple{k}{f}, \quad \forall f\in\Fh^I \right\}.
\end{equation}
Now, since the discrete space $\tilde{\Sigma}_h(\Th) $ does not require any kind of continuity between
elements, we force it by introducing the following computable discrete bilinear form
\begin{equation}
	c_h(\bs\bftau_h,\bs\nu_h)=-\sum_{E\in\Th} \int_{\partial E^I}\bs\nu_h \cdot \bs\tau_h \bbn~\df \quad \forall \bftau_h\in\tilde{\Sigma}_h(\Th) , \quad \forall \bs\nu_h \in\Lambda_h(\Fh^I),
\end{equation}
where $\partial E^I = \partial E \cap \Fh^I$ is the set of the internal faces of the element $E$.
Thus, the hybrid version of Problem~\eqref{problem:discrete} reads:
\begin{equation}\label{problem:discrete_hybrid}
	\left\{
	\begin{aligned}
        &\mbox{find } (\bfsigma_h,\bbu_h, \bs\lambda_h)\in\tilde{\Sigma}(\Th)\times U_h \times \Lambda_h(\Fh^I) \mbox{ s.t.}\\
		&a_h(\bfsigma_h,\bftau_h) + b(\bftau_h,\bbu_h) + c_h(\bftau_h,\bs\lambda_h) = 0  \quad &\forall \bftau_h\in\tilde{\Sigma}(\Th),\\
		&b(\bfsigma_h, \bbv_h) = -(\bbf,\bbv_h)  \quad &\forall \bbu_h\in U_h,\\
		&c_h(\bfsigma_h,\bs\nu_h) =0  \quad &\forall \bs\nu_h\in\Lambda_h(\Fh^I).
	\end{aligned}
\right.
\end{equation}
We observe that the two discrete Problems~\eqref{problem:discrete} and~\eqref{problem:discrete_hybrid} are equivalent. Indeed if $(\bfsigma_h,\bbu_h,\bflambda_h)\in\tilde{\Sigma}_{h}(\Th)\times U_h\times \Lambda_h(\Fh^I)$ solves Problem \eqref{problem:discrete_hybrid}, then  
$(\bfsigma_h,\bbu_h)\in{\Sigma}_{h}\times U_h$ and is the solution of Problem \eqref{problem:discrete}. Moreover, if $(\bfsigma_h,\bbu_h)\in{\Sigma}_{h}\times U_h$ is the solution of Problem \eqref{problem:discrete}, then there is a unique $\bflambda_h\in\Lambda_h(\Fh^I)$ such that $(\bfsigma_h,\bbu_h,\bflambda_h)$ is the solution of Problem \eqref{problem:discrete_hybrid}.
Moreover, since $U_h$ and $\tilde{\Sigma}_{h}(\Th)$ are now discontinuous, it is possible to apply the second step of the hybridization, the static condensation, which reduces the computational cost.
\subsection{Post-processing procedure}\label{subsection:POSTPROCESSINGHRFAMILY}
We present the post processing procedure to achieve a new discrete solution for the displacement field with an enhanced accuracy. More precisely, we exploit the information derived by the discrete solution of Problem~\eqref{problem:discrete_hybrid}, as the Lagrange multipliers $\bs\lambda_h$ and the discrete solution $\bbu_h$, to construct a non-conforming VEM approximation $\bbu_h^{*}$, see~\cite{AyusoLipnikovManzini} for more details about non-conforming VEM.
Let $\left[H^1(\Th)\right]^3$ be the broken $H^1$ vector space on $\Th$ defined as
\begin{equation}
	\left[H^1(\Th)\right]^3:= \prod_{E\in\Th} \left[H^1(E)\right]^3=\left\{\bbv \in \left[L^2(\O)\right]^3 : \bbv_{|E}\in \left[H^1(E)\right]^3 \right\}
\end{equation}
and endowed with the corresponding broken seminorm and norm
\begin{equation}
	|\bbv|^2_{1,\Th}:=\sum_{E \in \Th}\norm[0,E]{\nabla\bbv}^2, \qquad \norm[1,\Th]{\bbv}^2:=\sum_{E \in \Th}\norm[1,E]{\bbv}^2.
\end{equation}
Now, fixed an integer $l\geq 0$, we indicate the global non-conforming Sobolev space associated with a polyhedral decomposition $\Th$ 
\begin{equation}
	H^{1,nc}_l(\Th):=\left\{\bbv \in \left[H^1(\Th)\right]^3 \,: \, \int_e\ljump \bbv \rjump \bbq~\df =0 \quad  \forall \bbq\in\PolyTriple{l}{f}, \ \forall f \in\Fh  \right\}
\end{equation}
where 
\begin{equation}
	\ljump \bbv \rjump:=\left\{
	\begin{aligned}
		&\bbv^{+}\otimes\bbn_{E^+}+\bbv^{-}\otimes\bbn_{E^-} \quad &\text{ on } f\in \Fh^I\\
		&\bbv\otimes\bbn_{f}  & \text{ on }f\in \Fh^B,
	\end{aligned}
	\right.
\end{equation}
where $\otimes$ denotes the usual tensor product of vectors. Moreover, for each internal face $f\in\Fh^I$, we denote by $E^\pm$ the two elements that share the face $f$, and we write $\bbn_{E^{+}}$, $\bbn_{E^{-}}$ for the exterior normal of $f$ on $\partial E^{+}$ and $\partial E^{-}$, respectively.
\subsubsection{Non-conforming Virtual Element Methods}
Given a polyhedron $E\in\Th$, for an integer $l\geq 2$ 
we define the local non-conforming virtual space as
\begin{equation}~\label{eq:nonConformingSpace}
	\begin{aligned}
		U_h^*(E):=
		\left\{ 
		\vhstar\in \left[H^1(E)\right]^{3} \, :  \, \right. &\left.\Delta \vhstar\in\PolyTriple{l-2}{E}, \right.
		\\  
		& \left.
		\nabla\vhstar\bbn  \in \PolyTriple{l-1}{f}\quad \forall f\in\partial E
		\right\}.
	\end{aligned}
\end{equation}
\begin{remark}
We observe that the non-confoming space definition also holds for $l=1$, but for the aim of this section, we will always take $l=k+1\geq2$.
\end{remark}
Accordingly, for the local spaces $U_h^*(E)$, we can take the following degrees of freedom:
\begin{itemize}
	\item all the moments of $\vhstar$ of order up to $l-1$ on each face $f\in\partial E$:
	\begin{equation}~\label{eq:dofsNonConformingVEM_boundary}
		\vhstar \rightarrow \dfrac{1}{|f|} \int_f \vhstar \cdot \bbq_{l-1}~\df \quad \forall \bbq_{l-1} \in\left[\mathbb{P}_{l-1}(f)\right]^3;
	\end{equation}
	\item all the moments of $\vhstar$ of order up to $l-2$ on element $E$:
	\begin{equation}~\label{eq:dofsNonConformingVEM_internal}
		\vhstar \rightarrow \dfrac{1}{|E|} \int_E \bbv_h^* \cdot \bbq_{l-2}~\dEl \quad \forall \bbq_{l-2} \in\left[\mathbb{P}_{l-2}(E)\right]^3.
	\end{equation}
\end{itemize}
Therefore, we infer that the dimension of the space is 
\begin{equation}
	\dim(U_h^*(E))=3n_f^E\pi_{l-1,f} + 3\pi_{l-2,E}.
\end{equation}
The unisolvence of the degrees of freedom defined in is given by the following proposition, whose proof can be found in~\cite{AyusoLipnikovManzini}.
\begin{prop}~\label{prop:h_unisolvenceDofsNC}
	Let $E$ be a simple polyhedron with $n_f^E$ faces, and let $U^*_h(E)$ be the space defined in~\eqref{eq:nonConformingSpace}. The degrees of freedom~\eqref{eq:dofsNonConformingVEM_boundary} and~\eqref{eq:dofsNonConformingVEM_internal} are unisolvent for $U^*_h(E)$.
\end{prop}
We define the projection operator $\Pi^{\nabla}_{E} : U_h^*(E)\rightarrow \PolyTriple{l}{E}$ such that $\forall \vhstar \in U_h^*(E)$ we have 
\begin{equation}\label{eq:projectionPostProcessing1}
	\begin{aligned}
		\int_E \nabla \Pi^{\nabla}_E\vhstar \colon \nabla\bbq~\dEl = \int_E \nabla \vhstar \colon \nabla\bbq~\dEl, \quad \forall \bbq\in\PolyTriple{l}{E},
	\end{aligned}
\end{equation}
together with the condition
\begin{equation}\label{eq:projectionPostProcessing2}
		\int_E \Pi^{\nabla}_E\vhstar~\dEl = \int_E \vhstar~\dEl.
\end{equation}
Note that $\Pi_E^{\nabla}\vhstar$ is computable for any $\vhstar\in U_h^*(E)$ from the degrees of freedom~\eqref{eq:dofsNonConformingVEM_boundary} and~\eqref{eq:dofsNonConformingVEM_internal} since 
\begin{equation}
	\begin{aligned}
		\int_E \nabla \Pi^{\nabla}_E\vhstar \colon \nabla\bbq~\dEl = 
		-\int_E \vhstar\cdot \Delta \bbq~\dEl + \int_{\partial E} \vhstar \cdot \nabla \bbq\,\bbn ~\df.
	\end{aligned}
\end{equation}
The global non-conforming virtual element space is given by the standard gluing of the local approximation spaces, see~\eqref{eq:nonConformingSpace}
\begin{equation}
	U_h^*(\Th):=\left\{\bbv^*_h\in H^{1,nc}_l(\Th) \, :\, \bbv^*_{h_{|E}}\in U_h^*(E) \quad \forall\ E\in\Th  \right\}.
\end{equation}
The following results will be useful to prove the Theorem~\ref{theorem:postprocessing}.

\begin{prop}\label{prop:norm_and_seminorm_vhstar_estimate}
    Under assumption~\ref{meshA_1}, \ref{meshA_2} and \ref{meshA_3}, for every element $E\in\Th$ and every $\bbv_h^*\in U_h^*(E)$, it holds
    \begin{equation}~\label{eq:seminorm_vhstar}
        \seminorm[1,E]{\vhstar} \lesssim h_E^{-1} \norm[0,E]{\vhstar}
    \end{equation}
    and
    \begin{equation}~\label{eq:norm_vhstar}
        \norm[0,E]{\vhstar}\lesssim h_E^{1/2} \norm[0,\partial E]{\P^{l-1}_{\partial E}\vhstar}+ \norm[0,E]{\P^{l-2}_E\vhstar}.
    \end{equation}
\end{prop}
\begin{proof}
Let's start to prove~\eqref{eq:seminorm_vhstar}. Since $\vhstar\in\Uhstar$, we have that
\begin{equation}\label{eq:seminorm_vhstar1}
\begin{aligned}
\seminorm[1,E]{\vhstar}^2
&=-\int_E \Delta \vhstar \cdot\vhstar~\dEl + \int_{\partial E} \nabla \vhstar\bbn \cdot \vhstar~\df \\
&\leq \norm[0,E]{\Delta\vhstar}\norm[0,E]{\vhstar}+\norm[0,\partial E]{\nabla \vhstar\bbn}\norm[0,\partial E]{\vhstar}.
\end{aligned}
\end{equation}
By definition of $\Uhstar$ (cf.~\eqref{eq:nonConformingSpace}), we have that $\Delta\vhstar$ and  $(\nabla \vhstar\bbn)_{|\partial E}$ are two piecewise polynomial (vectorial) functions. So under the usual mesh assumptions~\ref{meshA_1}, \ref{meshA_2}, and~\ref{meshA_3}, the following inverse estimates
\begin{equation}\label{eq:inverseEstimateDelta}
    \norm[0,E]{\Delta\vhstar} \lesssim h_E^{-1} \seminorm[1,E]{\vhstar}
\end{equation}
and
\begin{equation}\label{eq:inverseEstimateNabla1}
    \norm[0,\partial E]{\nabla \vhstar\bbn} \lesssim h_E^{-1/2} \norm[-1/2,\partial E]{\nabla \vhstar\bbn}
\end{equation}
hold true, see~\cite[Lemma 6.3]{BLRXX} for the first estimate, whereas see~\cite{ARTIOLI2017155} for the second one. Therefore, using the technique developed in~\cite{BLRXX} and~\eqref{eq:inverseEstimateDelta}, we get
\begin{equation}\label{eq:inverseEstimateNabla2}
    \begin{aligned}
        \norm[0,\partial E]{\nabla \vhstar\bbn} 
        &\lesssim h_E^{-1/2} \norm[-1/2,\partial E]{\nabla \vhstar\bbn}\\
        &\lesssim h_E^{-1/2} \left(\norm[0,E]{\nabla \vhstar} + h_E\norm[0,E]{\Delta\vhstar}\right)\\
        &\lesssim h_E^{-1/2} \seminorm[1,E]{\vhstar}
    \end{aligned}
\end{equation}
Hence, combining~\eqref{eq:inverseEstimateDelta} and~\eqref{eq:inverseEstimateNabla2} into~\eqref{eq:seminorm_vhstar1}, we get
\begin{equation}
\begin{aligned}
\seminorm[1,E]{\vhstar}
&\lesssim h_E^{-1}\norm[0,E]{\vhstar} + h_E^{-1/2}\norm[0,\partial E]{\vhstar}.
\end{aligned}
\end{equation}
Exploiting the following trace inequality (see~\cite[Theorem 1.6.6]{Brenner-Scott:2008}) 
\begin{equation}
    \norm[0,\partial E]{\vhstar} \lesssim \norm[0,E]{\vhstar}^{1/2}\left[ \left(\seminorm[1,E]{\vhstar}^2 + h_E^{-2}\norm[0,E]{\vhstar}^2 \right)^{1/2}\right]^{1/2}
\end{equation}
we get 
\begin{equation}
\begin{aligned}
\seminorm[1,E]{\vhstar} &\lesssim h_E^{-1} \norm[0,E]{\vhstar}+h_E^{-1/2} \norm[0,E]{\vhstar}^{1/2} \left(\seminorm[1,E]{\vhstar}^2 + h_E^{-2}\norm[0,E]{\vhstar}^2 \right)^{1/4}\\
 &\lesssim h_E^{-1} \norm[0,E]{\vhstar}+
    h_E^{-1/2} \norm[0,E]{\vhstar}^{1/2}\seminorm[1,E]{\vhstar}^{1/2} + h_E^{-1} \norm[0,E]{\vhstar}\\
        &\lesssim h_E^{-1} \norm[0,E]{\vhstar}+
    h_E^{-1/2} \norm[0,E]{\vhstar}^{1/2}\seminorm[1,E]{\vhstar}^{1/2}.
\end{aligned}
\end{equation}
The Young's inequality with $\varepsilon$ ($\varepsilon>0$) applied to the second term of the previous inequality, gives
\begin{equation}
\seminorm[1,E]{\vhstar} \lesssim h_E^{-1} \norm[0,E]{\vhstar}+\frac{1}{2\varepsilon} h_E^{-1} \norm[0,E]{\vhstar}+\frac{\varepsilon}{2}\seminorm[1,E]{\vhstar},
\end{equation}
and choosing $\varepsilon$ sufficiently small, i.e., $\varepsilon=1/2$, we get~\eqref{eq:seminorm_vhstar}.
Now, we have to prove~\eqref{eq:norm_vhstar}. First of all, we split $\vhstar\in\Uhstar$ as 
\begin{equation}
    \vhstar = \left(\vhstar - \vhstarbar \right) + \vhstarbar = \whstar - \vhstarbar,
\end{equation}
where $\vhstarbar$ is the mean value of $\vhstar$ on E
\begin{equation}
\vhstarbar=\frac{1}{|E|}\int_E\vhstar~\dEl
\end{equation}
and $\whstar:=\vhstar -\vhstarbar$. Then, a direct computation shows that 
\begin{equation}\label{eq:triangular_inequality_vhstar}
    \norm[0,E]{\vhstar} \leq \norm[0,E]{\whstar} +\norm[0,E]{\vhstarbar} 
\end{equation}
Since $\whstar$ has zero mean value on $E$, using Poincar\'e estimate, see i.e. ~\cite{brenner2003,brenner2004}, we have
\begin{equation}\label{eq:poincare_inequality_vhstar}
    \norm[0,E]{\whstar} \lesssim h_E\seminorm[1,E]{\whstar}.
\end{equation}
As before, since $\left(\nabla \whstar \bbn\right)_{|_{\partial E}}$ and $\Delta \whstar$ are piecewise polynomial functions, we can use the estimates~\eqref{eq:inverseEstimateDelta} and~\eqref{eq:inverseEstimateNabla2} to obtain 
\begin{equation}
\begin{aligned}
\seminorm[1,E]{\whstar}^2 
&= -\int_E \Delta \whstar \cdot \whstar~\dEl + \int_{\partial E} \nabla\whstar\bbn \cdot \whstar~\df\\
&= -\int_E \Delta \whstar \cdot \P^{l-2}_E\whstar~\dEl + \int_{\partial E} \nabla\whstar\bbn \cdot \P^{l-1}_{\partial E}\whstar~\df\\
&\leq \norm[0,E]{\Delta \whstar}\norm[0,E]{\P^{l-2}_E\whstar} + \norm[0,\partial E]{\nabla\whstar\bbn} \norm[0,\partial E]{\P^{l-1}_{\partial E}\whstar}\\
&\lesssim h_E^{-1} \seminorm[1,E]{\whstar}\norm[0,E]{\P^{l-2}_E \whstar}+h_E^{-1/2}\seminorm[1,E]{\whstar}\norm[0,\partial E]{\P^{l-1}_{\partial E}\whstar}  
\end{aligned}
\end{equation}
Therefore, we get
\begin{equation}\label{eq:seminorm_whstar1}
\seminorm[1,E]{\whstar} \lesssim h_E^{-1} \norm[0,E]{\P^{l-2}_E\ \whstar}+h_E^{-1/2}\norm[0,\partial E]{\P^{l-1}_{\partial E}\whstar}.
\end{equation}
Combining~\eqref{eq:triangular_inequality_vhstar}, \eqref{eq:poincare_inequality_vhstar} and \eqref{eq:seminorm_whstar1}, we infer 
\begin{equation}~\label{eq:norm_vhstar1}
    \norm[0,E]{\vhstar} \lesssim h_E^{1/2} \norm[0,\partial E]{\P^{l-1}_{\partial E}\whstar} + \norm[0,E]{\P^{l-2}_E\whstar}+ \norm[0, E]{\vhstarbar}.
\end{equation}
The triangle inequality together the trace inequality give the following estimate of the first term of~\eqref{eq:norm_vhstar1}
\begin{equation}
\begin{aligned}
\norm[0,\partial E]{\P^{l-1}_{\partial E}\whstar}
&\leq \norm[0,\partial E]{\P^{l-1}_{\partial E}(\vhstar- \vhstarbar)} \\
&\leq \norm[0,\partial E]{\P^{l-1}_{\partial E}\vhstar} + \norm[0,\partial E]{\P^{l-1}_{\partial E}\vhstarbar}\\
&=  \norm[0,\partial E]{\P^{l-1}_{\partial E}\vhstar} + \norm[0,\partial E]{\vhstarbar}\\
&  \leq \norm[0,\partial E]{\P^{l-1}_{\partial E}\vhstar} + h^{-1/2}\norm[0, E]{\vhstarbar},
\end{aligned}
\end{equation}
and so~\eqref{eq:norm_vhstar1} becomes
\begin{equation}~\label{eq:norm_vhstar2}
    \norm[0,E]{\vhstar} \lesssim h_E^{1/2} \norm[0,\partial E]{\P^{l-1}_{\partial E}\vhstar} + \norm[0,E]{\P^{l-2}_E\whstar}+ \norm[0, E]{\vhstarbar}.
\end{equation}
Noticing that 
\begin{equation*}
    \int_{E} \P^{l-2}_E\whstar \cdot \vhstarbar~\dEl=0
\end{equation*}
so it holds
\begin{equation}~\label{eq:seminorm_whstar4}
    \norm[0,E]{\P^{l-2}_E\whstar} + \norm[0,E]{\vhstarbar} \lesssim \norm[0 E]{\P^{l-2}_E(\whstar+ \vhstarbar)} = \norm[0,E]{\P^{l-2}_E\vhstar}.
\end{equation}
Now estimate~\eqref{eq:norm_vhstar} follows from~\eqref{eq:norm_vhstar2} and~\eqref{eq:seminorm_whstar4}.
\end{proof}
Now we introduce the main result of this section.
\begin{thm}~\label{theorem:postprocessing}
Fixed an integer $k\geq 1$, which we remind to be the degree of accuracy of our VEM schemes presented in Section~\ref{section:VEMHRFAMILY}, let $\left(\bfsigma, \bbu\right)$ be the solution of the continuous Problem~\eqref{problem:cont-strong} and $\left(\bfsigma_h, \bbu_h, \bs\lambda_h\right)$ be the discrete solution of hybridized Problem~\eqref{problem:discrete_hybrid}. Let be $l=k+1$ and we define $\uhstar\in\Uhstarglob$ such that it holds
    \begin{equation}\label{eq:uhstar_condition_postprocessing}
    \left\{
        \begin{aligned}
          & \P^{k}_{\partial E}\left(\uhstar -\bs\lambda_h \right)=\bfzero,\\
          & \P_E^{k-1}\left(\uhstar - \bbu_h \right)=\bfzero.
        \end{aligned}
    \right.
    \end{equation}
Then we have the following estimate 
    \begin{equation}~\label{eq:convergence_estimate_postprocessing}
        \norm[0]{\bbu-\uhstar}\lesssim h^{k+2}.
    \end{equation}
\end{thm}
\begin{proof}
For the displacement field $\bbu$, we define the natural non-conforming interpolant $\uhstartilde\in\Uhstarglob$ such that 
\begin{equation}\label{eq:interpolant_condition_postprocessing}
    \left\{
        \begin{aligned}
          & \P^{k}_{\partial E}\left(\uhstartilde -\bbu \right)=\bfzero, \\
          & \P_E^{k-1}\left(\uhstartilde - \bbu \right)=\bfzero.
        \end{aligned}
    \right.
\end{equation}
We observe that both the function $\uhstar$ and the interpolant $\uhstartilde$ are well-defined for the unisolvence of the degrees of freedom, see Proposition \ref{prop:h_unisolvenceDofsNC} with $l=k+1$. Writing now
\begin{equation}\label{eq:decoposition_postprocessing}
    \bbu -\uhstar = \left(\bbu-\uhstartilde\right) +\left(\uhstartilde -\uhstar\right)
\end{equation}
and using the triangle inequality, we have that
\begin{equation}\label{eq:triangle_inequality_postprocessing}
    \begin{aligned}
        \norm[0]{\bbu-\uhstar}&\leq \norm[0]{\bbu-\uhstartilde}+\norm[0]{\uhstartilde-\uhstar}.
    \end{aligned}
\end{equation}
Using standard arguments, see i.e.~\cite{Brenner-Scott:2008}, we can estimate the first term as follows
\begin{equation}\label{eq:estimate_u_uhstar}
    \norm[0]{\bbu-\uhstartilde}\lesssim h^{k+2}.
\end{equation}
Otherwise, to estimate $\norm[0]{\uhstartilde-\uhstar}$,  we observe that  from~\eqref{eq:uhstar_condition_postprocessing} and~\eqref{eq:interpolant_condition_postprocessing} we have
\begin{equation}\label{eq:information_uhstar_uhstartilde}
    \left\{
        \begin{aligned}
          & \P^{k}_{\partial E}\left(\uhstar -\uhstartilde \right)=\P^{k}_{\partial E}\left(\bs\lambda_h -\bbu \right),\\
          & \P^{k-1}_E\left(\uhstar -\uhstartilde  \right)=\P^{k-1}_E\left(\bbu_h -\bbu  \right) =\P^{k-1}_E\left(\bbu_h -\P^{k}_E\bbu  \right) .
        \end{aligned}
    \right.
    \end{equation}
Then, always taking $l=k+1$, for each element $E\in\Th$ we can use estimate~\eqref{eq:norm_vhstar} of Proposition~\ref{prop:norm_and_seminorm_vhstar_estimate}, the continuity of the projection operator and we get
\begin{equation}
\begin{aligned}
\norm[0,E]{\uhstar-\uhstartilde}
&\lesssim h_E^{1/2} \norm[0,\partial E]{\P^{k}_{\partial E}(\uhstar-\uhstartilde)}+ \norm[0,E]{\P^{k-1}_E(\uhstar-\uhstartilde)}\\
&\lesssim h_E^{1/2} \norm[0,\partial E]{\P^{k}_{\partial E}(\bs\lambda_h -\bbu)}+ \norm[0,E]{\P^{k-1}_E(\bbu_h -\bbu)}\\
&\lesssim h_E^{1/2} \norm[0,\partial E]{\P^{k}_{\partial E}(\bs\lambda_h -\bbu)}+ \norm[0,E]{\P^{k-1}_E\left(\bbu_h -\P^{k}_E\bbu  \right)}\\
&\lesssim h_E^{1/2} \norm[0,\partial E]{\P^{k}_{\partial E}(\bs\lambda_h -\bbu)}+ \norm[0,E]{\bbu_h -\P^{k}_E\bbu}
\end{aligned}
\end{equation}
Now, exploiting Theorem~\ref{theorem:convergenceAndSuperconvergence}, the estimates for the Lagrange multipliers in~\cite[Theorem 1.4]{ArnoldBrezzi_1985} and summing all the local estimates, we get
\begin{equation}\label{eq:estimate_uhstar_interpolant}
    \norm[0]{\uhstar-\uhstartilde} \lesssim h^{k+2}.
\end{equation}
Estimate~\eqref{eq:convergence_estimate_postprocessing} now follows from~\eqref{eq:triangle_inequality_postprocessing}, \eqref{eq:estimate_u_uhstar} and \eqref{eq:estimate_uhstar_interpolant}.
\end{proof}
\graphicspath{{Figures/}}
\newcommand{\sizeGraph}{0.35}
\newcommand{\sizeMesh}{0.24}
\section{Numerical results}~\label{section:numerical_results}
In this section, we numerically assess the behavior of the proposed VEM schemes through some numerical examples where the analytical solution is available. We first show some convergence results then we exhibit the improvement that one has when applying the hybridization procedure. The numerical scheme is developed inside the vew++ library, a c++ code realized at the University Milano-Bicocca (https://sites.google.com/view/vembic/home).
\subsection{Test cases}
We consider the following two problems on the unit cube domain $\O=\left[ 0,1\right]^3$ and the materials are homogeneous and isotropic for all experiments.
\paragraph{\textbf{Test a: compressible material}}~\label{eq:test_a} We consider an elastic problem with the following exact displacement solution
\begin{equation}
	\begin{aligned}
	\bbu = (10\,S(x,y,z),10\,S(x,y,z),10\,S(x,y,z))^T,
	\end{aligned}
\end{equation}
and loading term
\begin{equation}
	\begin{aligned}
	\bbf =
	\begin{pmatrix}
	 -10\pi^2 ((\lambda + \mu) \cos(\pi x) \sin(\pi y + \pi z) - (\lambda + \mu)S(x, y, z))\\
	 -10\pi^2 ((\lambda + \mu) \cos(\pi y) \sin(\pi x + \pi z) - (\lambda + \mu)S(x, y, z))\\
	 -10\pi^2 ((\lambda + \mu) cos(\pi z) \sin(\pi x + \pi y) - (\lambda + \mu)S(x, y, z))
	\end{pmatrix}
	\end{aligned}
\end{equation}
where $S(x, y, z) = \sin(\pi x) \sin(\pi y) \sin(\pi z)$. For this problem, we impose homogeneous natural boundary conditions and the Lam\`e constants are set as $\lambda =1$ and $\mu=1$.
\paragraph{\textbf{Test b: nearly incompressible material}}~\label{eq:test_b}
The elastic problem has the following displacement solution:
\begin{eqnarray*}
\left\{
\begin{array}{l}
u_1={\sin(2\pi x)}^2\left(\cos(2\pi y)\sin(2\pi y){\sin(2\pi z)}^2-\cos(2\pi z)\sin(2\pi z){\sin(2\pi y)}^2 \right)\\
u_2={\sin(2\pi y)}^2\left(\cos(2\pi z)\sin(2\pi z){\sin(2\pi x)}^2-\cos(2\pi x)\sin(2\pi x){\sin(2\pi z)}^2 \right)\\
u_3={\sin(2\pi z)}^2\left(\cos(2\pi x)\sin(2\pi x){\sin(2\pi y)}^2-\cos(2\pi y)\sin(2\pi y){\sin(2\pi x)}^2 \right)
\end{array}.
\right .
\end{eqnarray*}
As before the loading term $\bbf$ is computed accordingly. For this problem we consider the following Lam\`e coefficients: $\lambda = 10^5$ and $\mu=0.5$. Also for this example we consider homogeneous natural boundary conditions.
\subsection{Meshes}
We discretize our domain in four different ways, see Fig.~\ref{fig:meshes}:
%
%
\begin{itemize}
	\item \texttt{Cube}, a uniform mesh composed by structured cubes;
	\item \texttt{Tetra},  standard tetrahedral mesh built via the Delaunay criterion~\cite{tetgen};
	\item \texttt{CVT}, a Voronoi tasselation optimized by the Lloyd algorithm~\cite{CVT1999}
	\item \texttt{Rand}, a Voronoi tassellation achieved with random control points and without optimization.
\end{itemize}
\begin{figure}[ht]
	\centering
	\renewcommand{\thesubfigure}{}
	\subfigure[\texttt{Cube}]{\includegraphics[width=\sizeMesh\textwidth,trim = 0mm 0mm 0mm 0mm, clip]{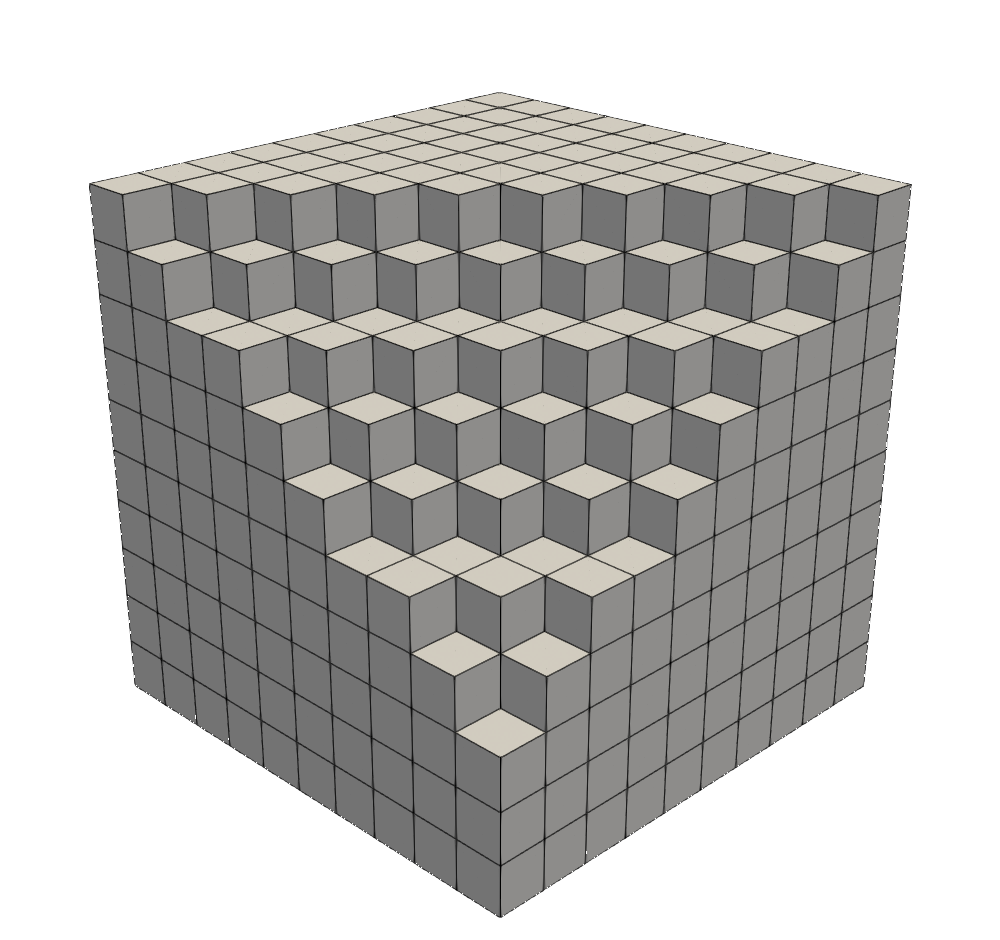}}
	\subfigure[\texttt{Tetra}]{\includegraphics[width=\sizeMesh\textwidth,trim = 0mm 0mm 0mm 0mm, clip]{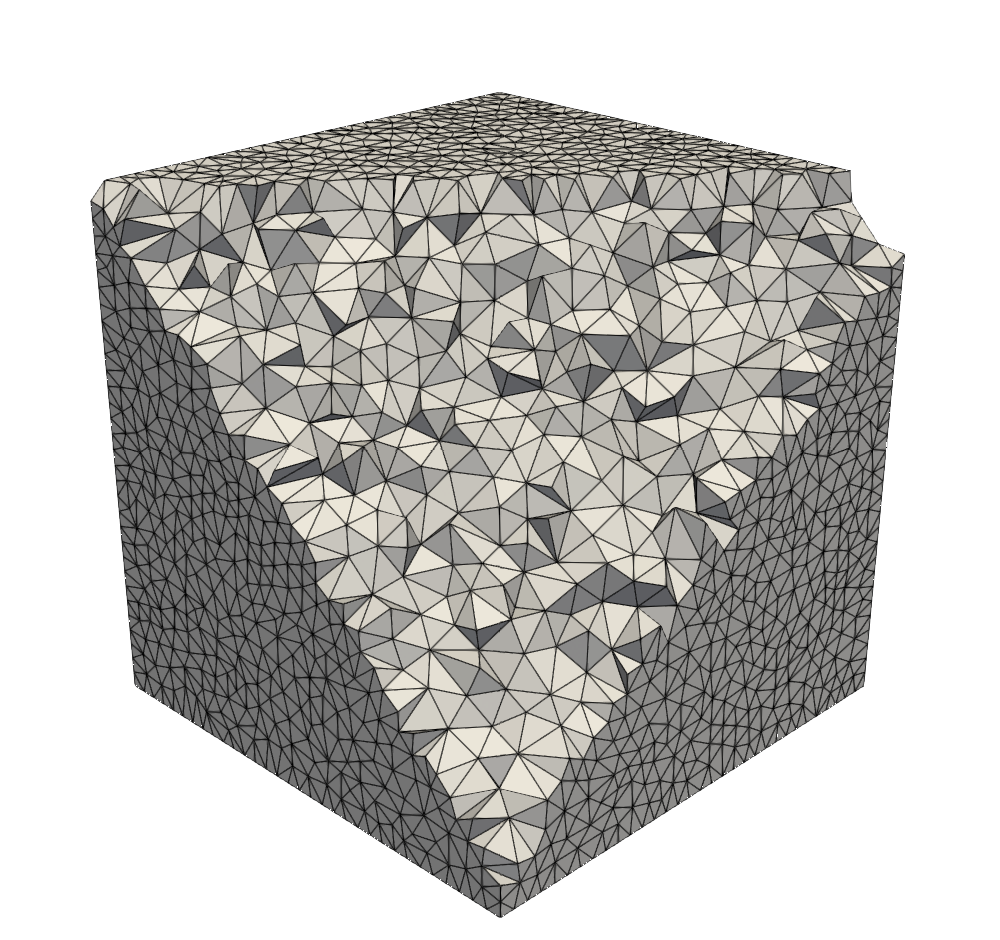}}
	\subfigure[\texttt{CVT}]{\includegraphics[width=\sizeMesh\textwidth,trim = 0mm 0mm 0mm 0mm, clip]{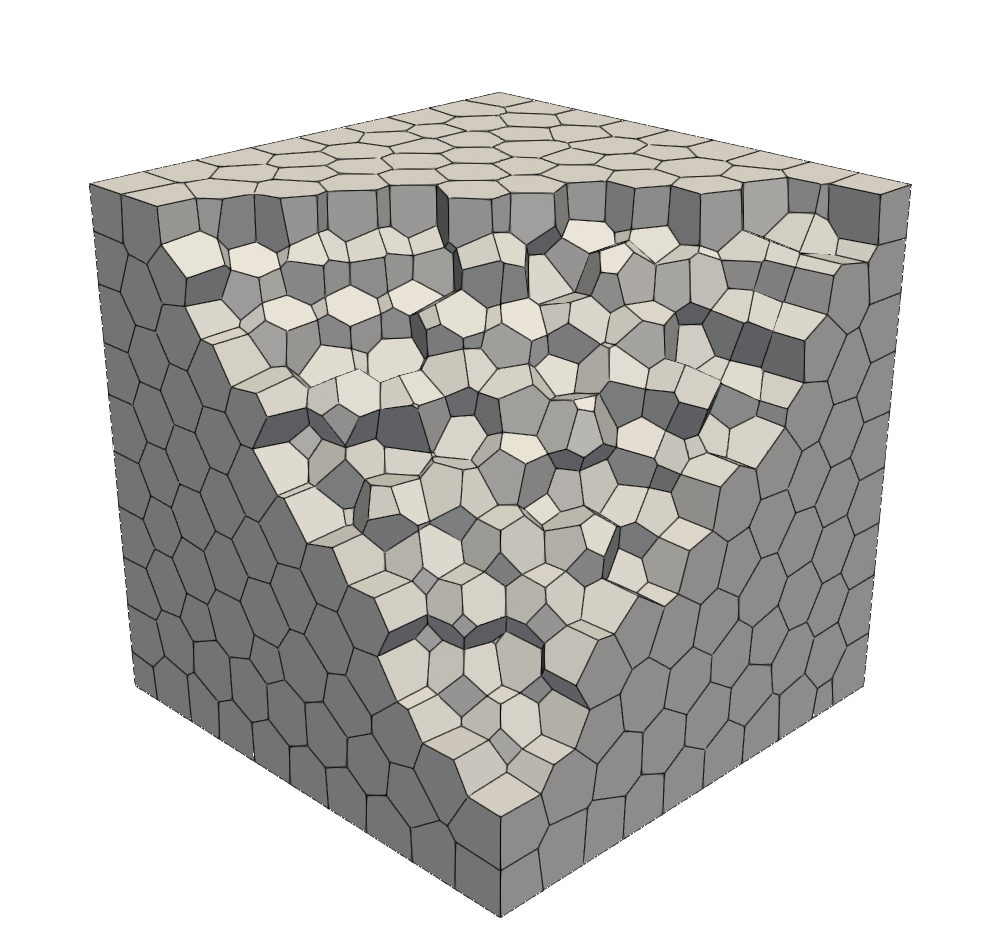}}
	\subfigure[\texttt{Rand}]{\includegraphics[width=\sizeMesh\textwidth,trim = 0mm 0mm 0mm 0mm, clip]{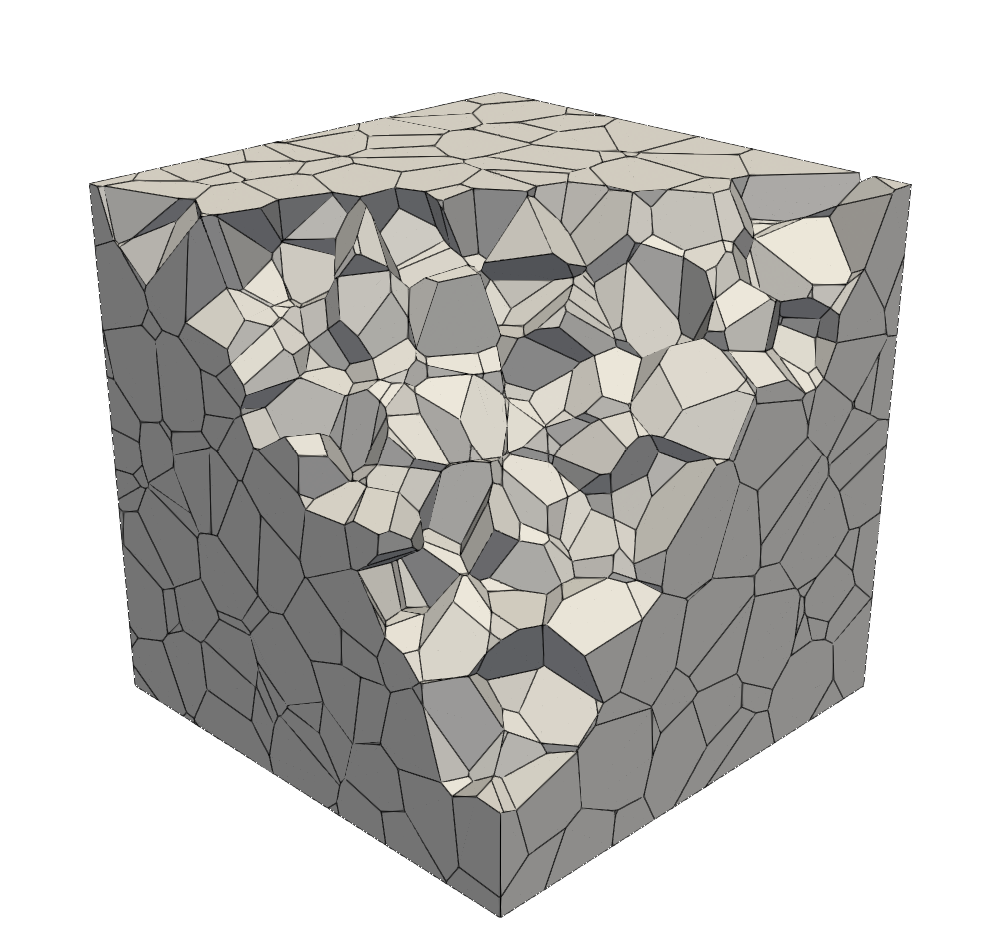}}
	\caption{Overview of adopted meshes for convergence assessment numerical tests.}
	\label{fig:meshes}
\end{figure}
The meshes taken into account have two different levels of complexity. The first two meshes, \texttt{Cube} and \texttt{Tetra}, are very standard and they are composed by regular shaped elements, i.e., high quality tetrahedrons and standard cubes. Instead, the last two meshes, \texttt{CVT} and \texttt{Rand}, represent an interesting challenge for the robustness of our approach. Indeed, they have elements with small faces or edges, and we remark that such a situation is not covered by the developed theory, i.e., assumptions~\ref{meshA_1}, \ref{meshA_2} and \ref{meshA_3}.
To verify the convergence rate, for each type of mesh, we take the following mesh-size $h$:
\begin{equation*}
    h=\frac{1}{N_E}\sum_{E\in\Th} h_E,
\end{equation*}
where $N_E$ is the number of the elements in the mesh $~\Th$ and $h_E$ is the diameter of the polyhedron $E$.
\subsection{Convergence Analysis} 
Let $(\bfsigma,\bbu)$ and $(\bfsigma_h,\bbu_h)$ be the continuous and discrete VEM solution of our elasticity problem. In order to analyse the accuracy and the convergence rate, we compute the following error indicators:
\begin{itemize}
	\item[$\bullet$] $E_{\bbu} :=\norm[0]{\bbu - \bbu_h}$	
 \item[$\bullet$]
    $E_{\bfsigma, \bdiv}:=\norm[0]{\bdiv \bfsigma - \bdiv \bfsigma_h}$
 \item[$\bullet$] $E_{\bfsigma, \Pi}:=\norm[0]{\bfsigma - \Pi_E\bfsigma_h}$
 \item[$\bullet$] {Discrete error for the stress field:}
	\begin{equation*}
	E_{\bfsigma, \partial}   :=\sqrt{\sum_{f \in \Fh} h_f\int_{f} \kappa\, | (\bfsigma - \bfsigma_h)\bbn |^2},
	\end{equation*}
	where $\kappa=\frac{1}{2} {\rm tr}(\D)$ (the material is homogeneous).
\end{itemize}
For each test we build a sequence of four meshes with decreasing mesh size parameter $h$ and the trend of each error indicator is computed and compared to the expected convergence trend, which, for sufficiently regular data is $O(h^{k+1})$ in accordance to estimate~\eqref{eq:error_estimates} in Theorem~\ref{theorem:convergenceAndSuperconvergence}.
More precisely, since the discrete stress functions are virtual we use the last three errors ($	E_{\bfsigma, \bdiv}$, $E_{\bfsigma, \Pi}$, $E_{\bfsigma, \partial}$) to show the convergence for the stress field.

\begin{figure}[ht]
		\centering
		\subfigure[]{\includegraphics[width=\sizeGraph\textwidth,trim = 0mm 0mm 0mm 0mm, clip]{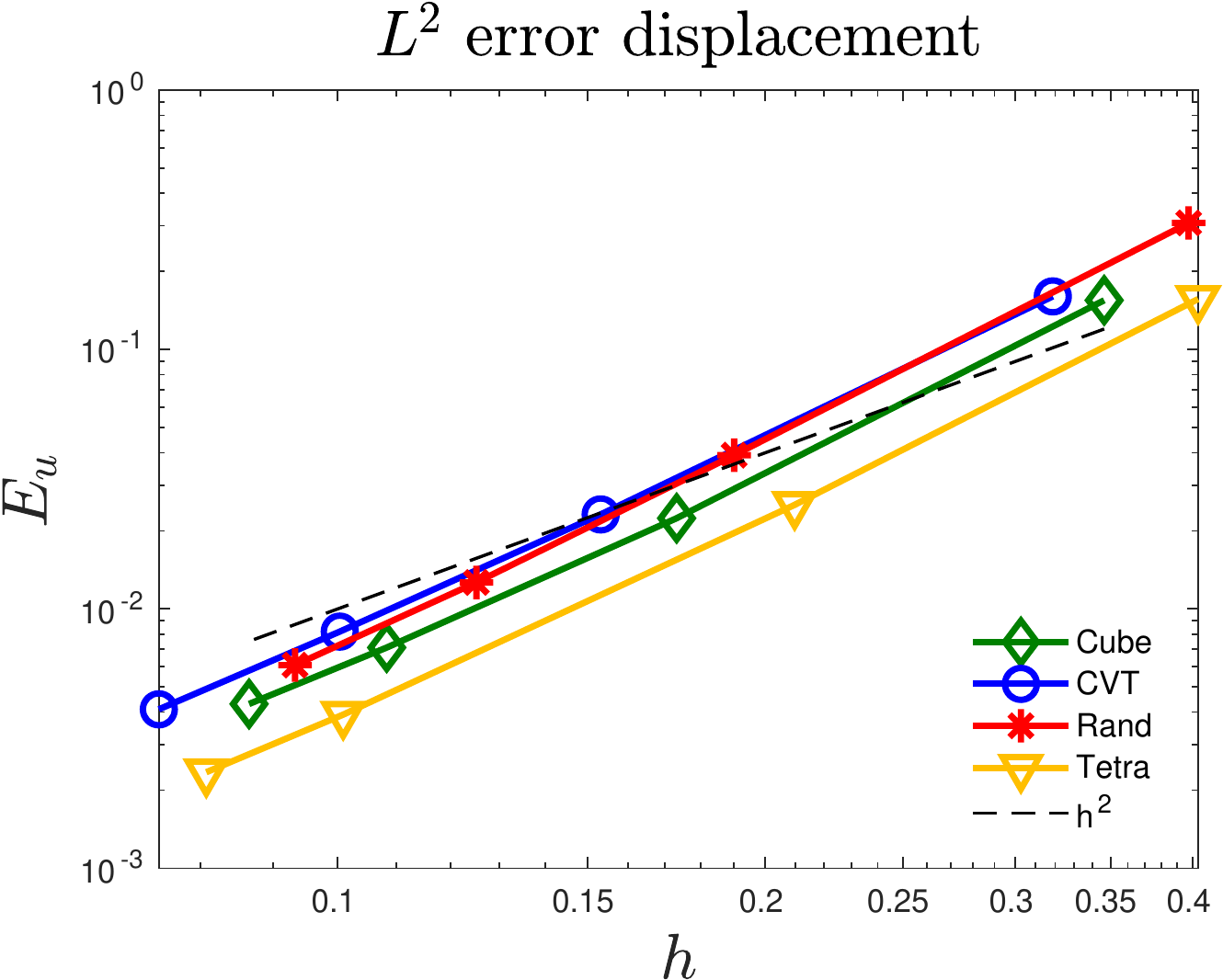}}
		\subfigure[]{\includegraphics[width=\sizeGraph\textwidth,trim = 0mm 0mm 0mm 0mm, clip]{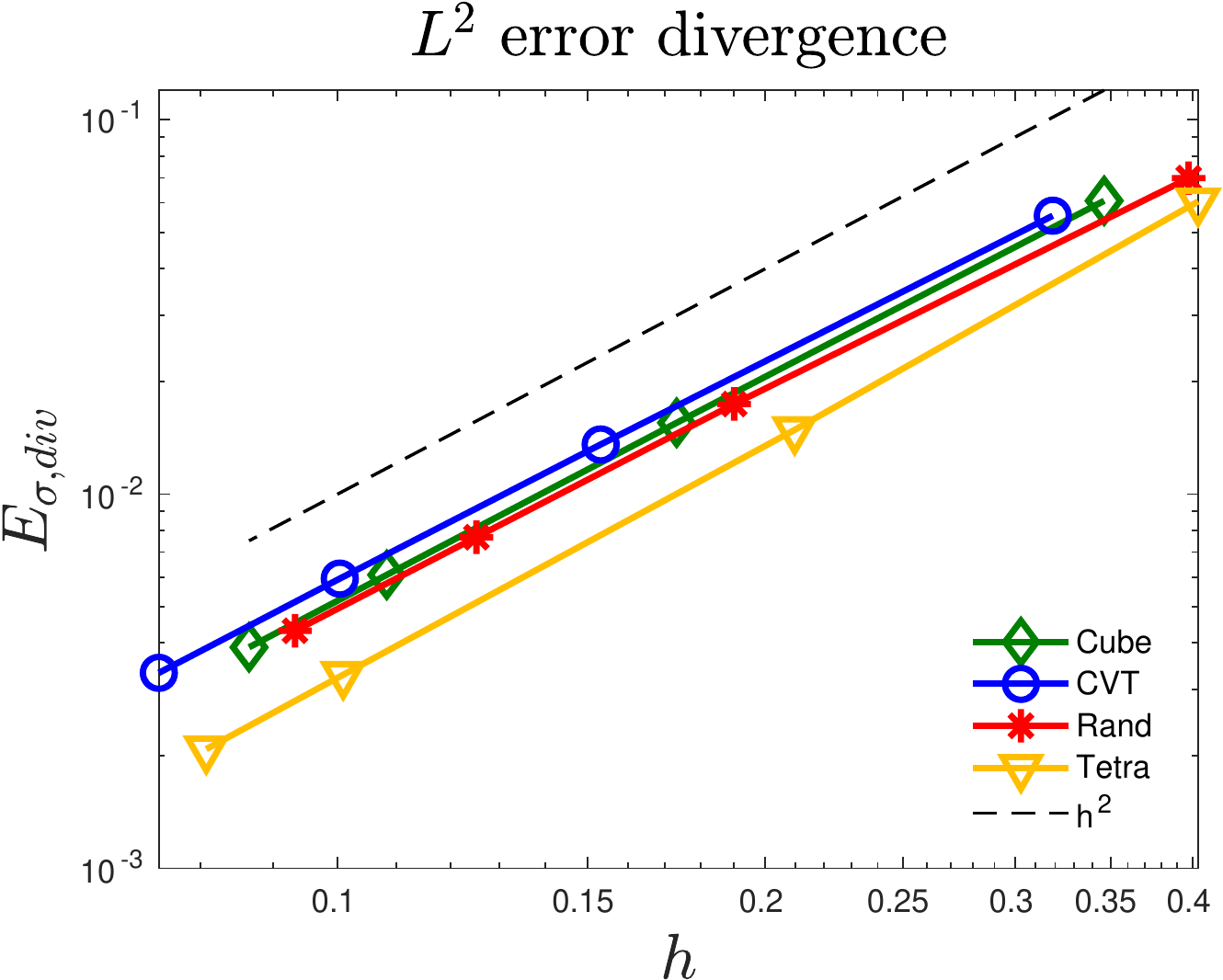}}\\
		\subfigure[]{\includegraphics[width=\sizeGraph\textwidth,trim = 0mm 0mm 0mm 0mm, clip]{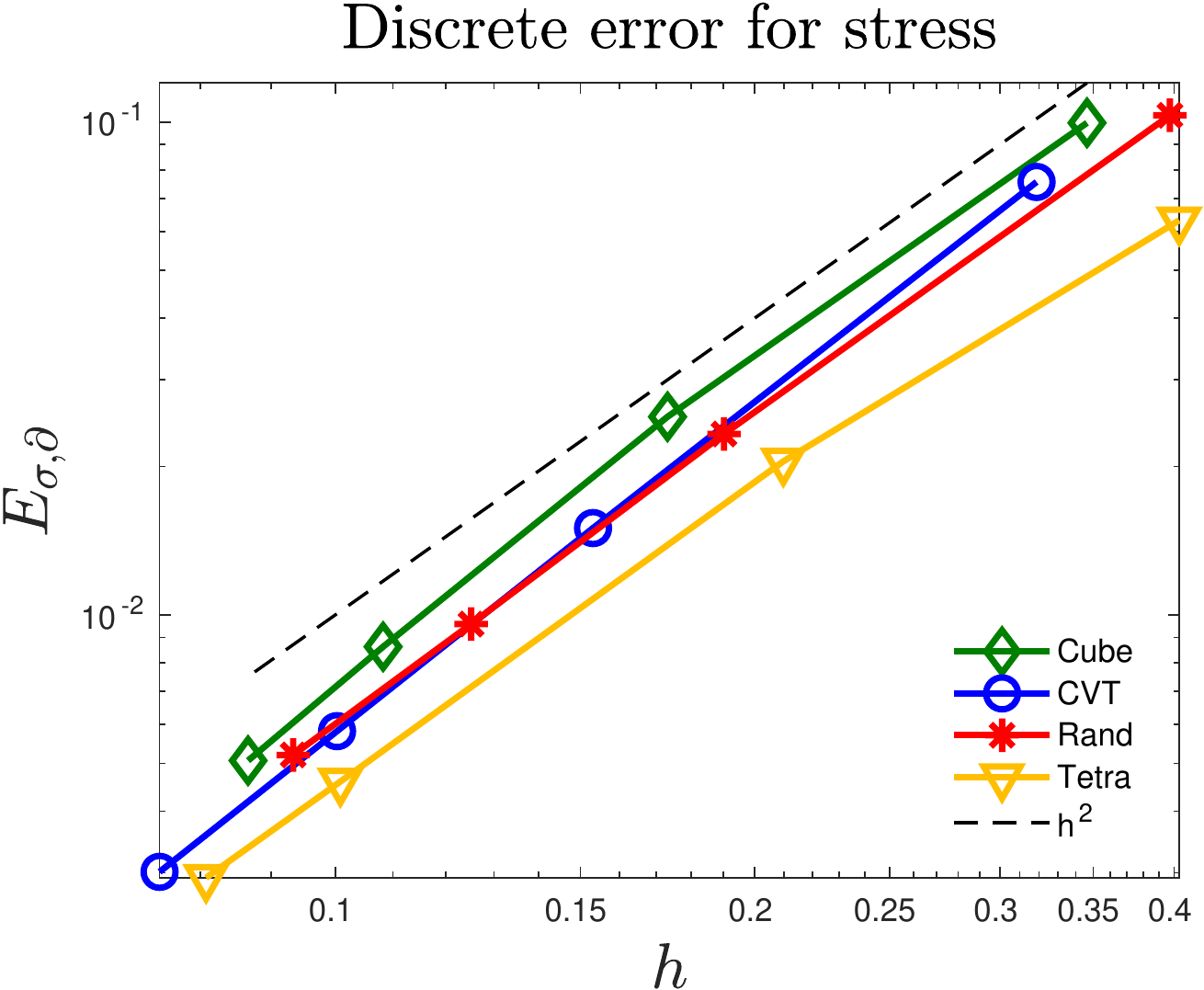}}
		\subfigure[]{\includegraphics[width=\sizeGraph\textwidth,trim = 0mm 0mm 0mm 0mm, clip]{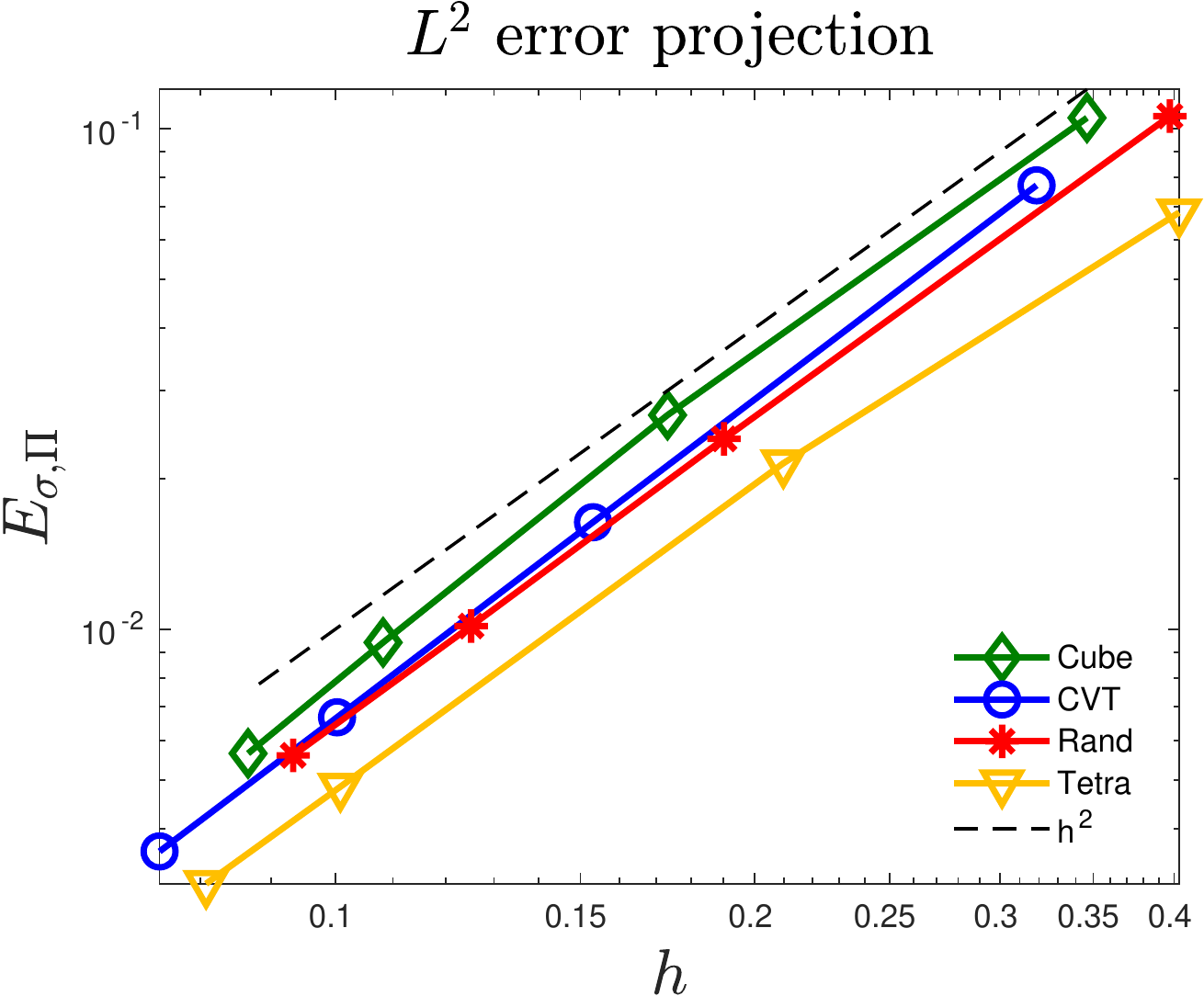}}\\
		\caption{\textbf{Test a: compressible material.} $h$-convergence results for all meshes for k=1.\label{fig:resuTest1}}
\end{figure}
 	\begin{figure}[ht]
	\centering
	\subfigure[]{\includegraphics[width=\sizeGraph\textwidth,trim = 0mm 0mm 0mm 0mm, clip]{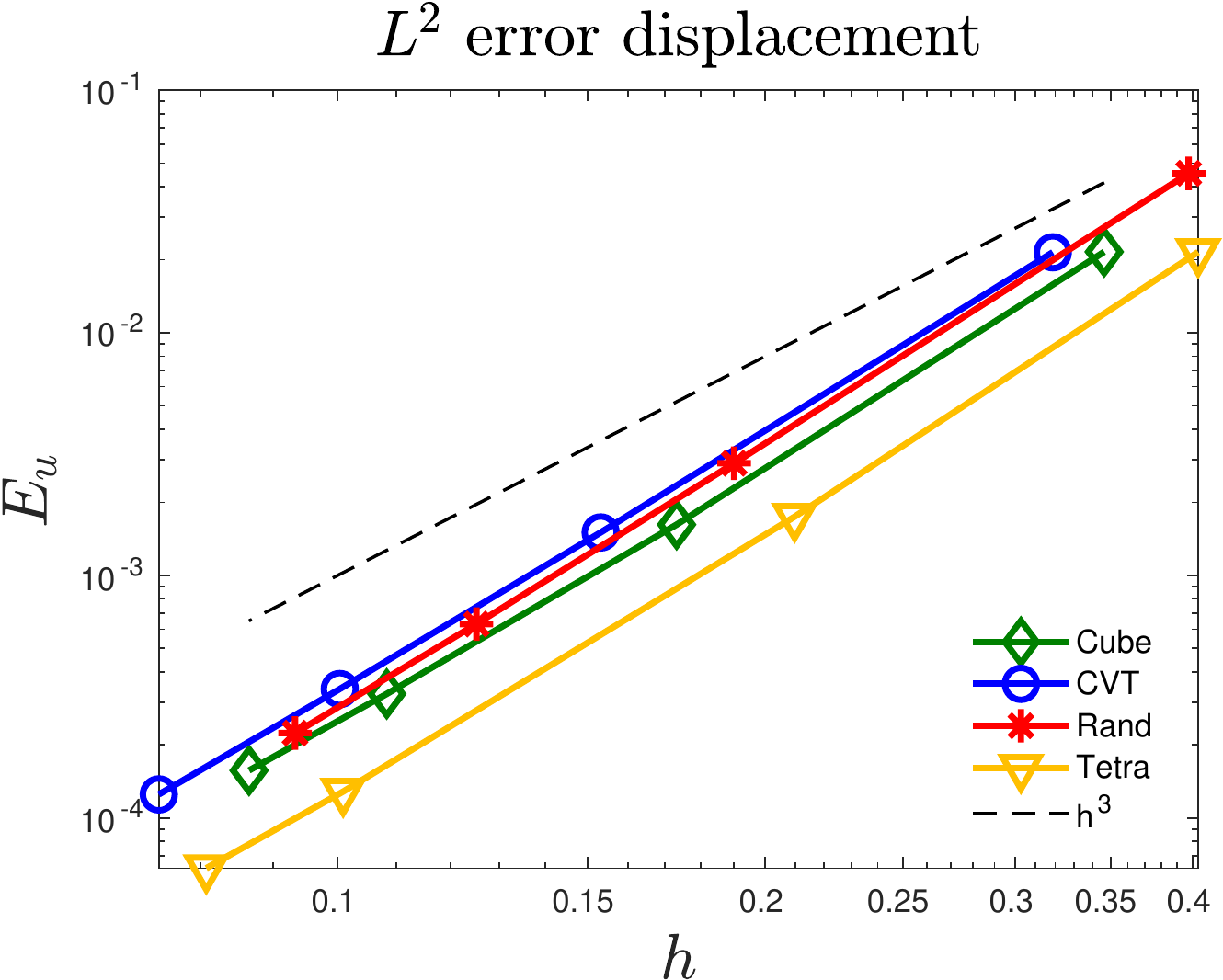}}
		\subfigure[]{\includegraphics[width=\sizeGraph\textwidth,trim = 0mm 0mm 0mm 0mm, clip]{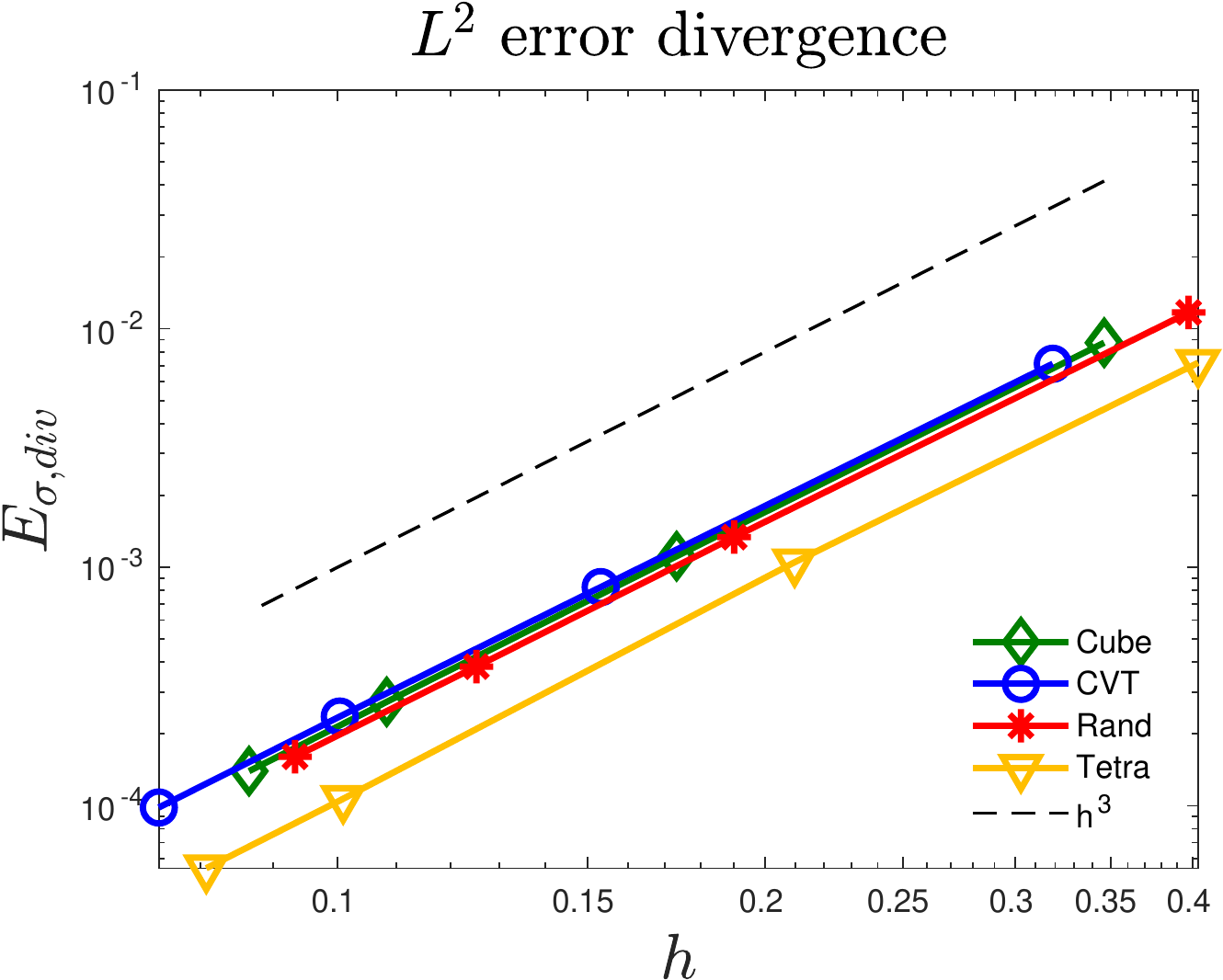}}\\
	\subfigure[]{\includegraphics[width=\sizeGraph\textwidth,trim = 0mm 0mm 0mm 0mm, clip]{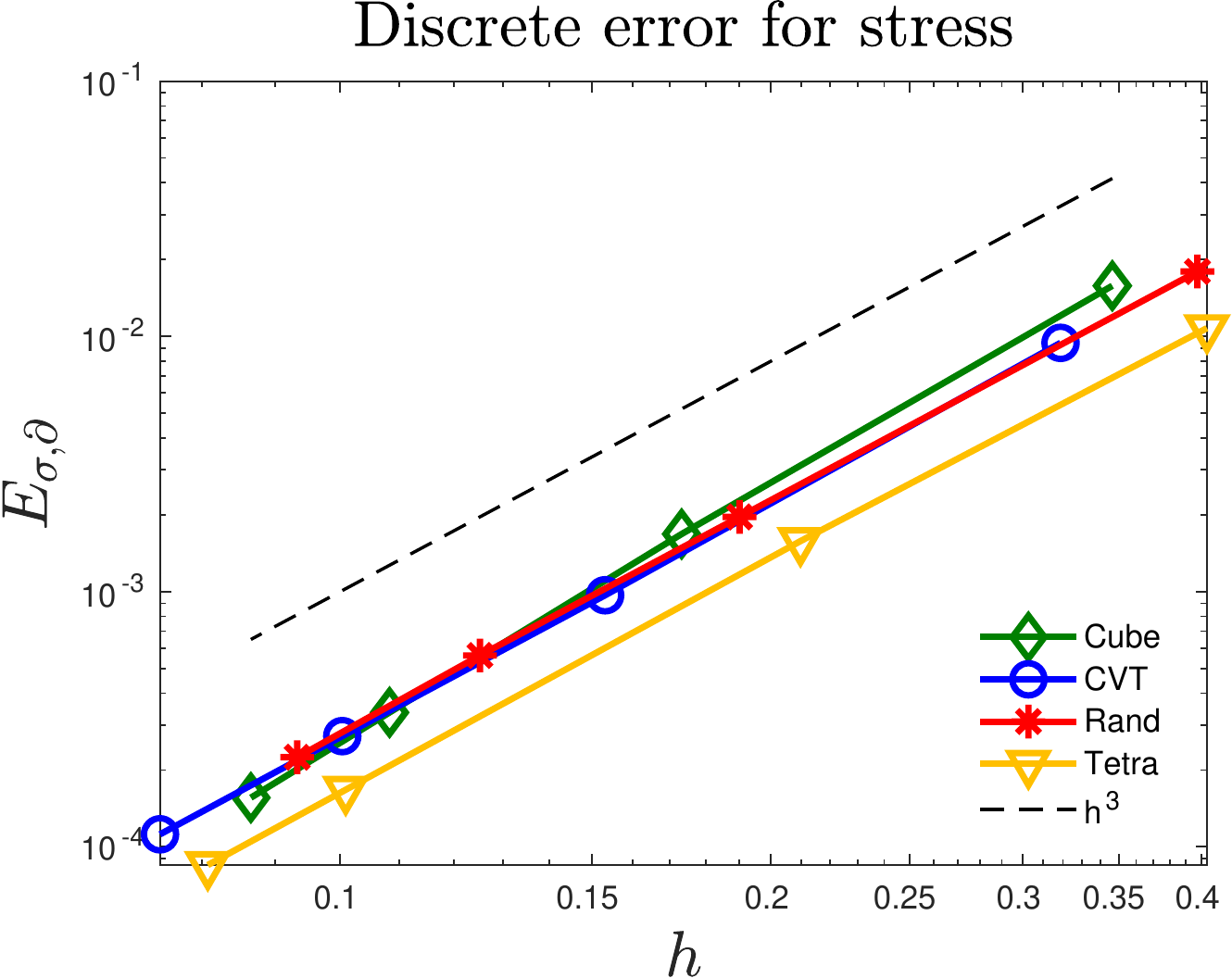}}
	\subfigure[]{\includegraphics[width=\sizeGraph\textwidth,trim = 0mm 0mm 0mm 0mm, clip]{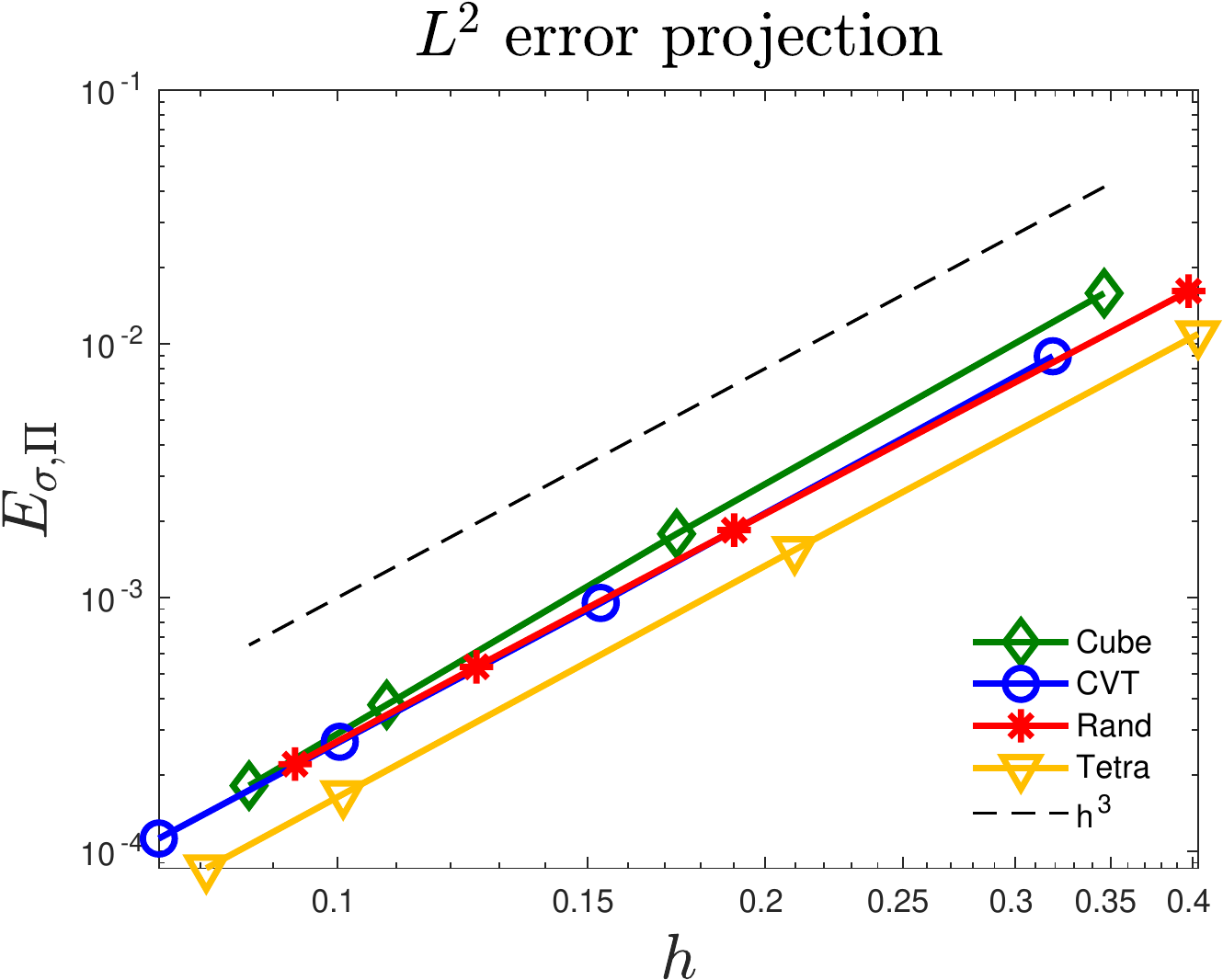}}\\
	\caption{\textbf{Test a: compressible material.} $h$-convergence results for all meshes for k=2.\label{fig:resuTest2}}
\end{figure}
\begin{figure}[ht]
\centering
\subfigure[]{\includegraphics[width=\sizeGraph\textwidth,trim = 0mm 0mm 0mm 0mm, clip]{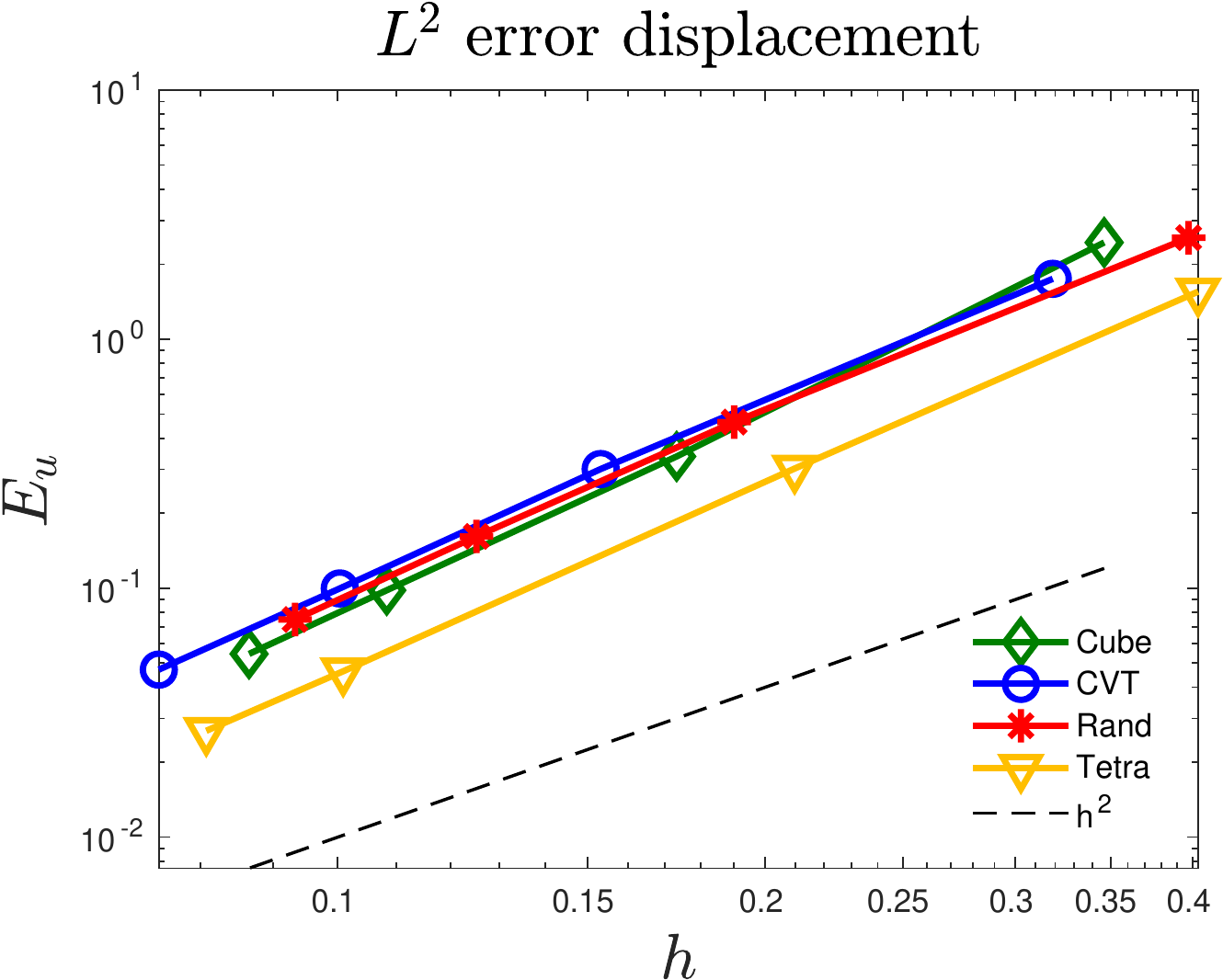}}
\subfigure[]{\includegraphics[width=\sizeGraph\textwidth,trim = 0mm 0mm 0mm 0mm, clip]{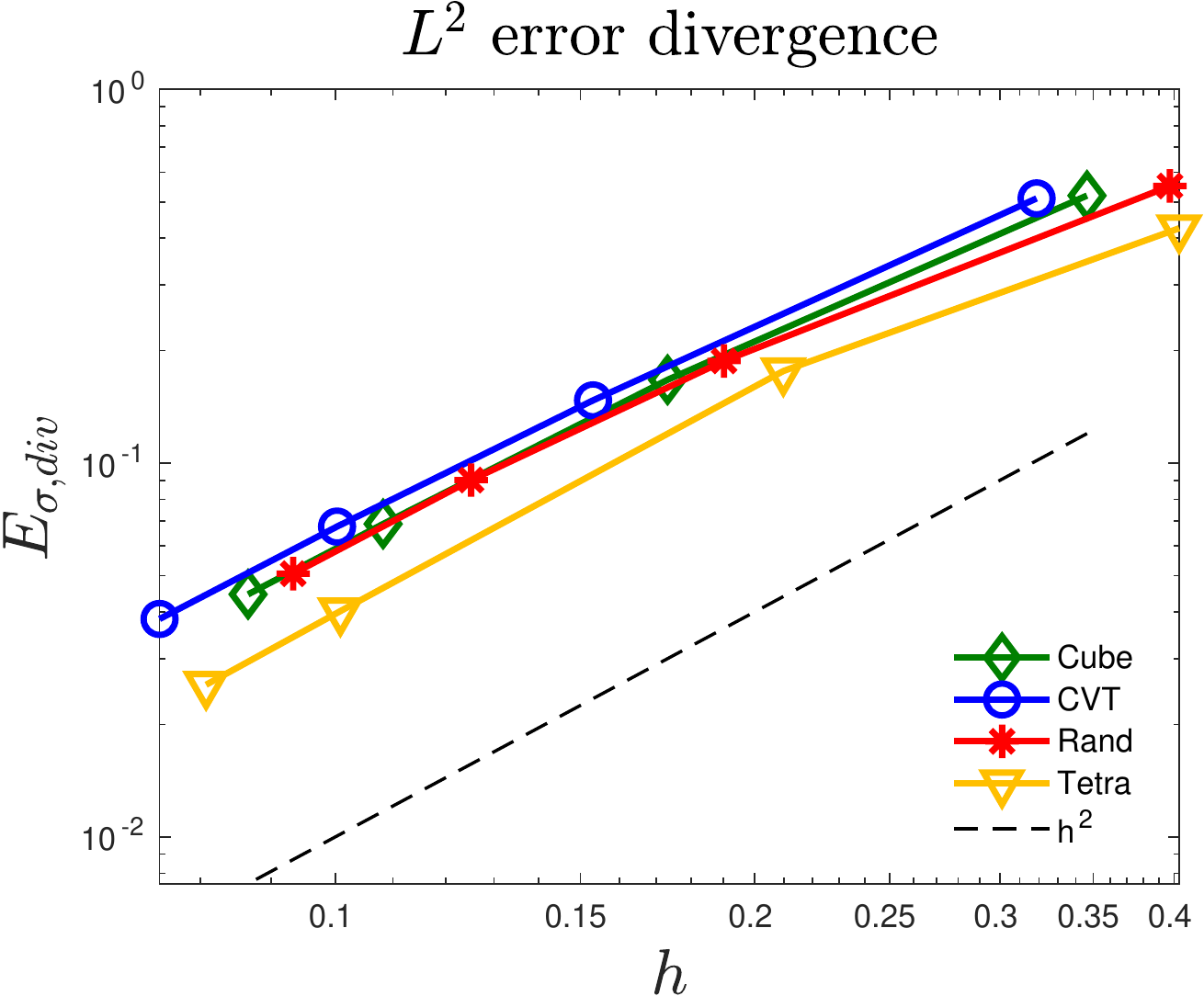}}\\
\subfigure[]{\includegraphics[width=\sizeGraph\textwidth,trim = 0mm 0mm 0mm 0mm, clip]{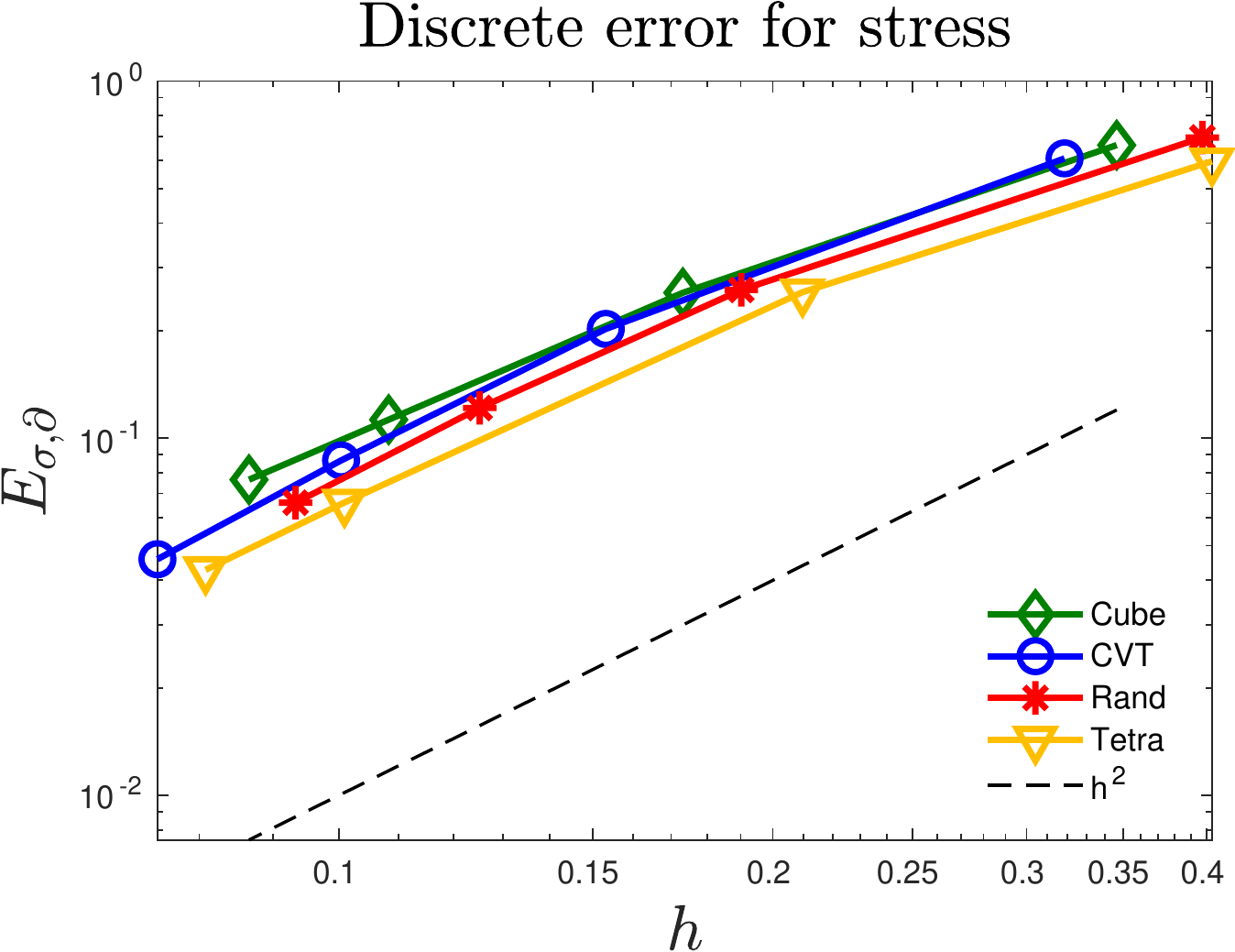}}
	\subfigure[]{\includegraphics[width=\sizeGraph\textwidth,trim = 0mm 0mm 0mm 0mm, clip]{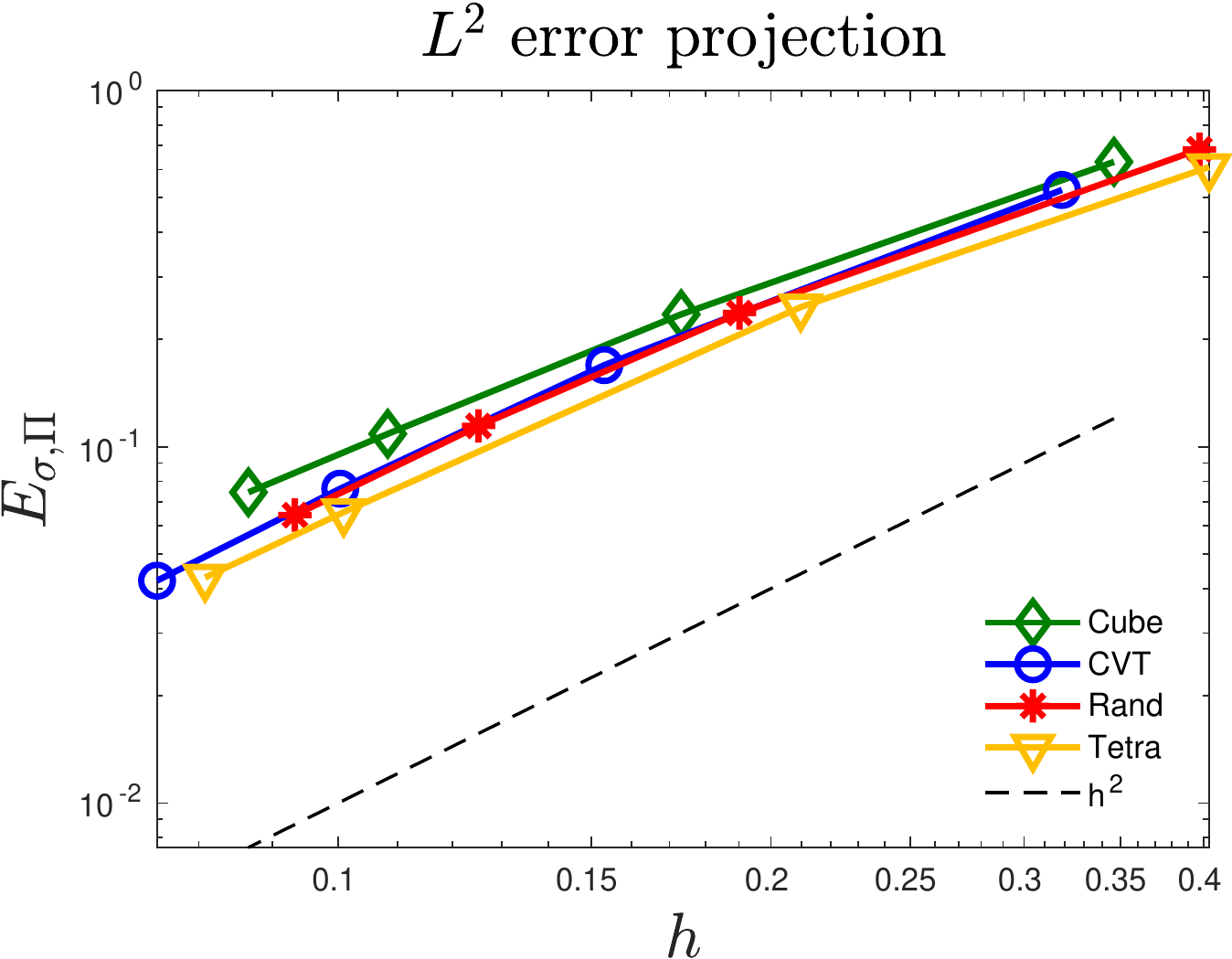}}\\
\caption{\textbf{Test b: nearly incompressible material.} $h$-convergence results for all meshes for k=1.\label{fig:resuTest3}}
\end{figure}
\begin{figure}[ht]
\centering
\subfigure[]{\includegraphics[width=\sizeGraph\textwidth,trim = 0mm 0mm 0mm 0mm, clip]{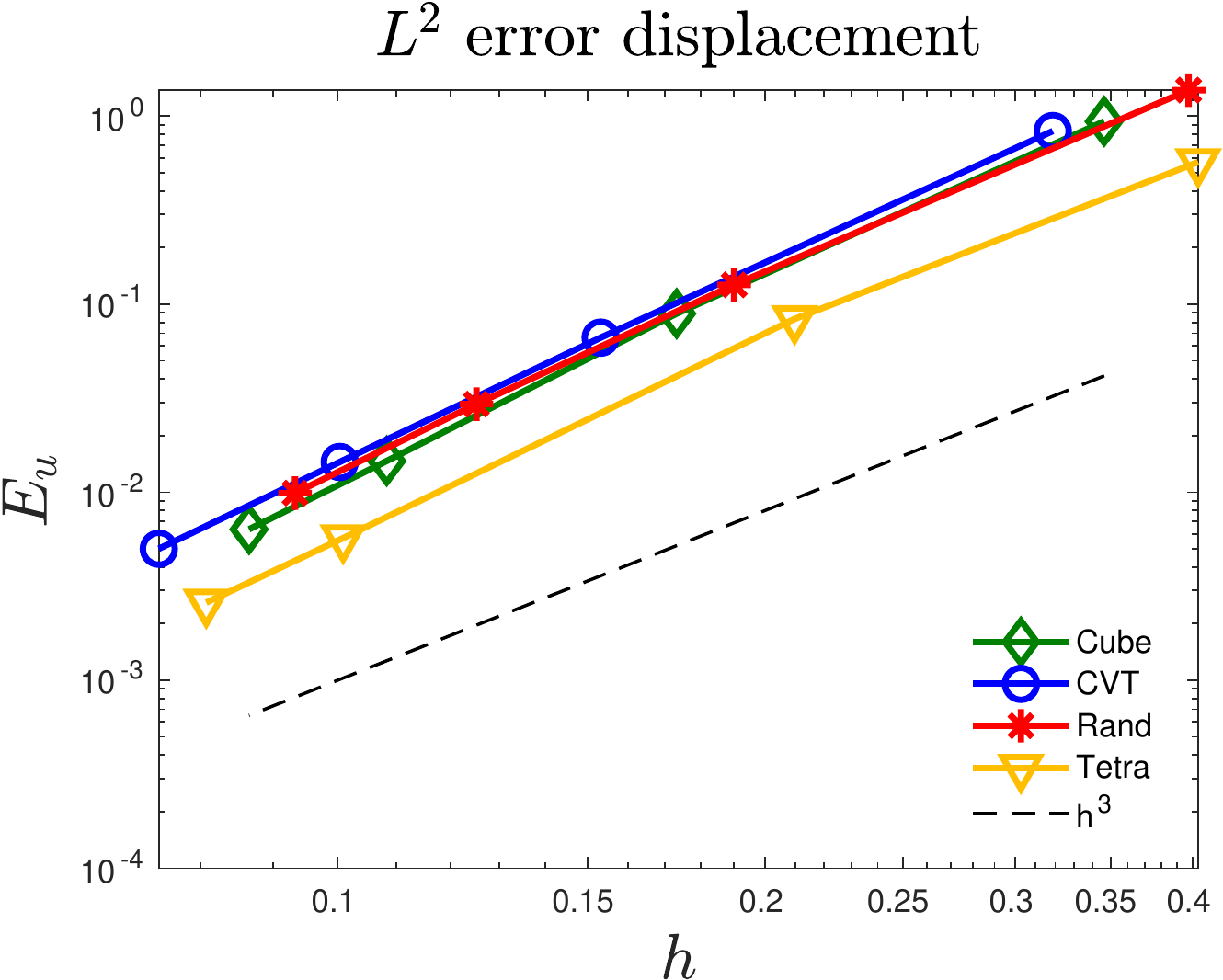}}
\subfigure[]{\includegraphics[width=\sizeGraph\textwidth,trim = 0mm 0mm 0mm 0mm, clip]{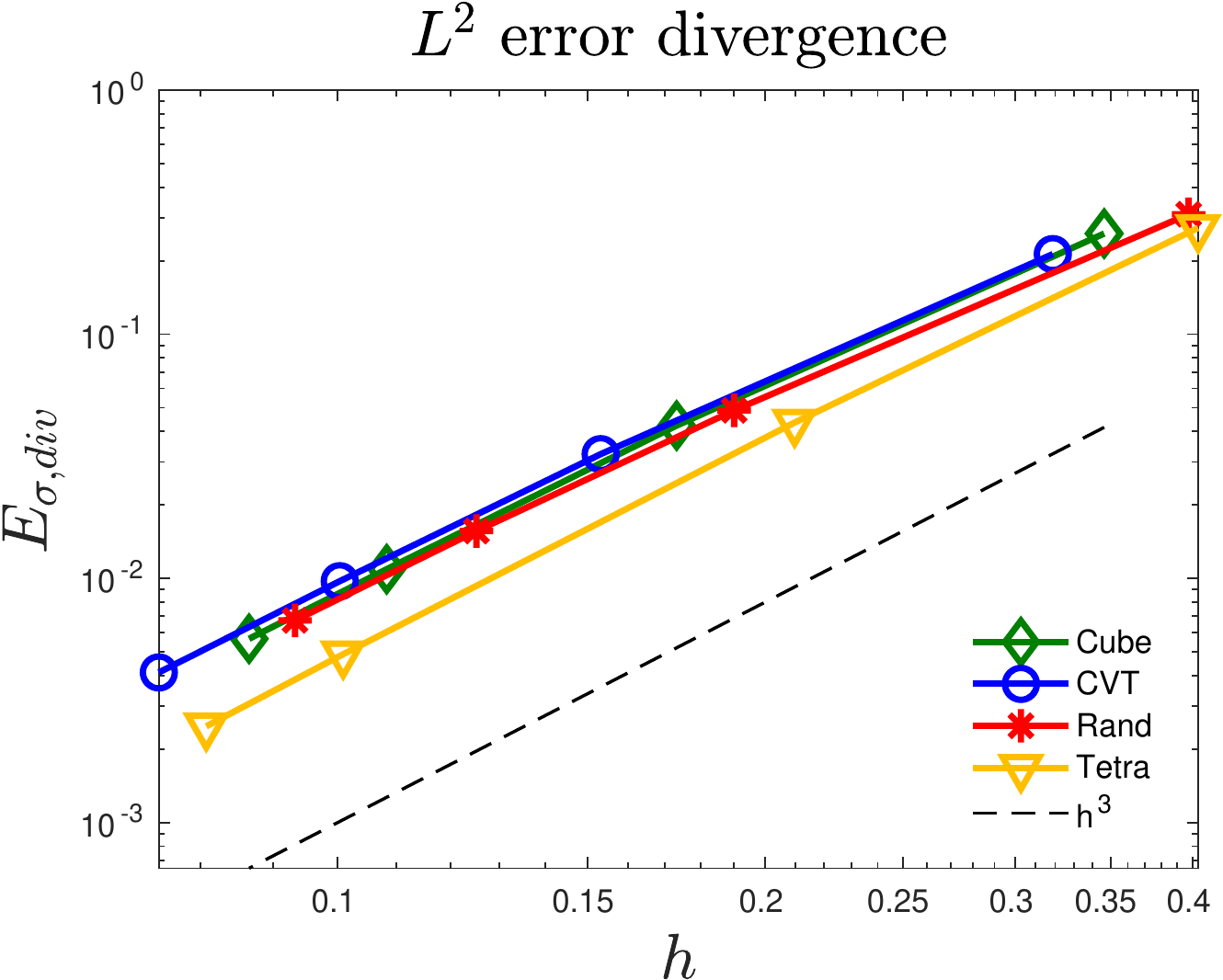}}\\
\subfigure[]{\includegraphics[width=\sizeGraph\textwidth,trim = 0mm 0mm 0mm 0mm, clip]{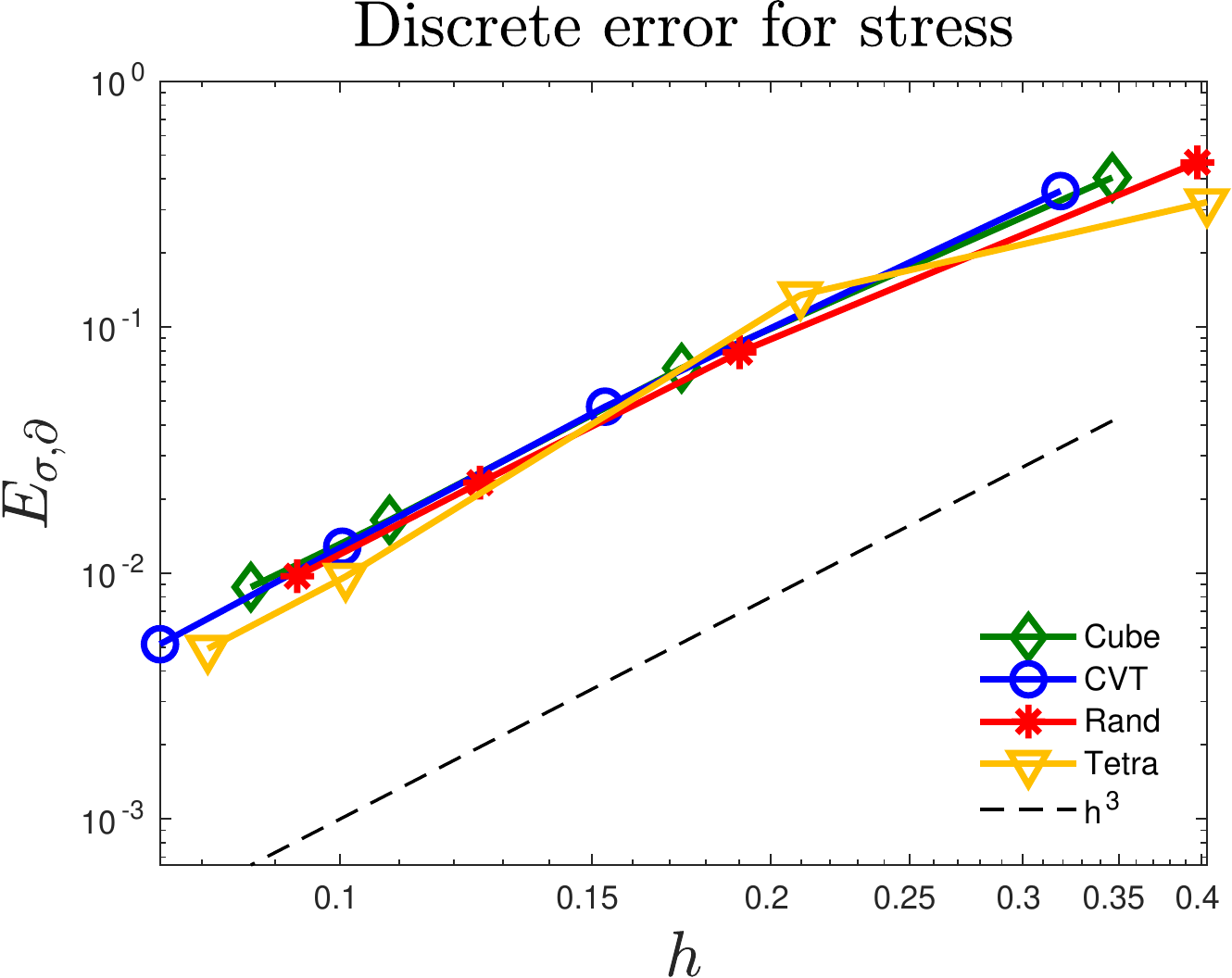}}
\subfigure[]{\includegraphics[width=\sizeGraph\textwidth,trim = 0mm 0mm 0mm 0mm, clip]{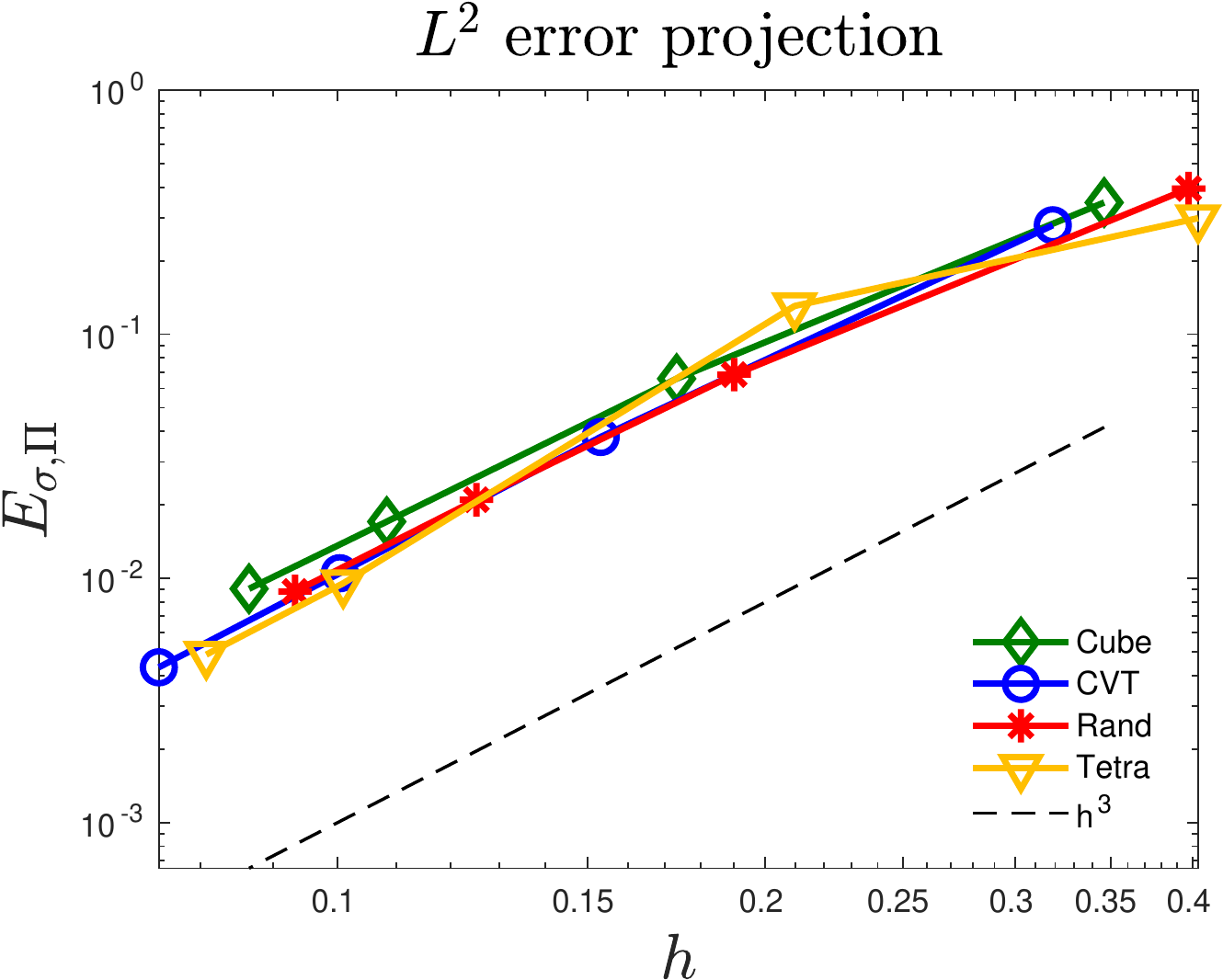}}\\
\caption{\textbf{Test b: nearly incompressible material.} $h$-convergence results for all meshes for k=2\label{fig:resuTest4}}
\end{figure}

Figures \ref{fig:resuTest1} and \ref{fig:resuTest2} report the $h$-convergence of the proposed VEM approach for \textbf{Test a} when $k=1,2$, respectively. Figures \ref{fig:resuTest3} and \ref{fig:resuTest4} the convergence for \textbf{Test b} always when $k=1,2$. The asymptotic convergence rate is approximately equal to the right order for all error norms and meshes. Moreover, the convergence lines are close to each others and this fact confirms the robustness of the proposed schemes with respect to element shape. 
\subsection{Post-processing results}
In this section, we numerically confirm the superconvergence result, predicted by Theorem~\ref{theorem:convergenceAndSuperconvergence} and used in the proof of Theorem~\ref{theorem:postprocessing}, and we exhibit the accuracy of our post-processing procedure. For sake of simplicity, we will consider only the compressible case (\textbf{Test a}). We consider the following error quantities:
\begin{itemize}
    \item the error norm $E_{\P_h^k \bbu}:=\norm[0]{\P^k_h\bbu - \bbu_h}$ for the superconvergence result. According to Theorem~\ref{theorem:convergenceAndSuperconvergence} the expected behavior of such an error is $O(h^{k+2})$ for sufficiently regular problems;
    \item the error norm $E_{\uhstar}:=\norm[0]{\bbu -\Pi^{\nabla} \uhstar}$ for the post-processed displacement. Since $\uhstar$ is virtual inside the element, we use the above error where  $\Pi^{\nabla}$  is the projection operator defined by~\eqref{eq:projectionPostProcessing1} and~\eqref{eq:projectionPostProcessing2}. Also this quantity behaves as $O(h^{k+2}$
\end{itemize}

\begin{figure}[ht]
\centering
\subfigure[]{\includegraphics[width=\sizeGraph\textwidth,trim = 0mm 0mm 0mm 0mm, clip]{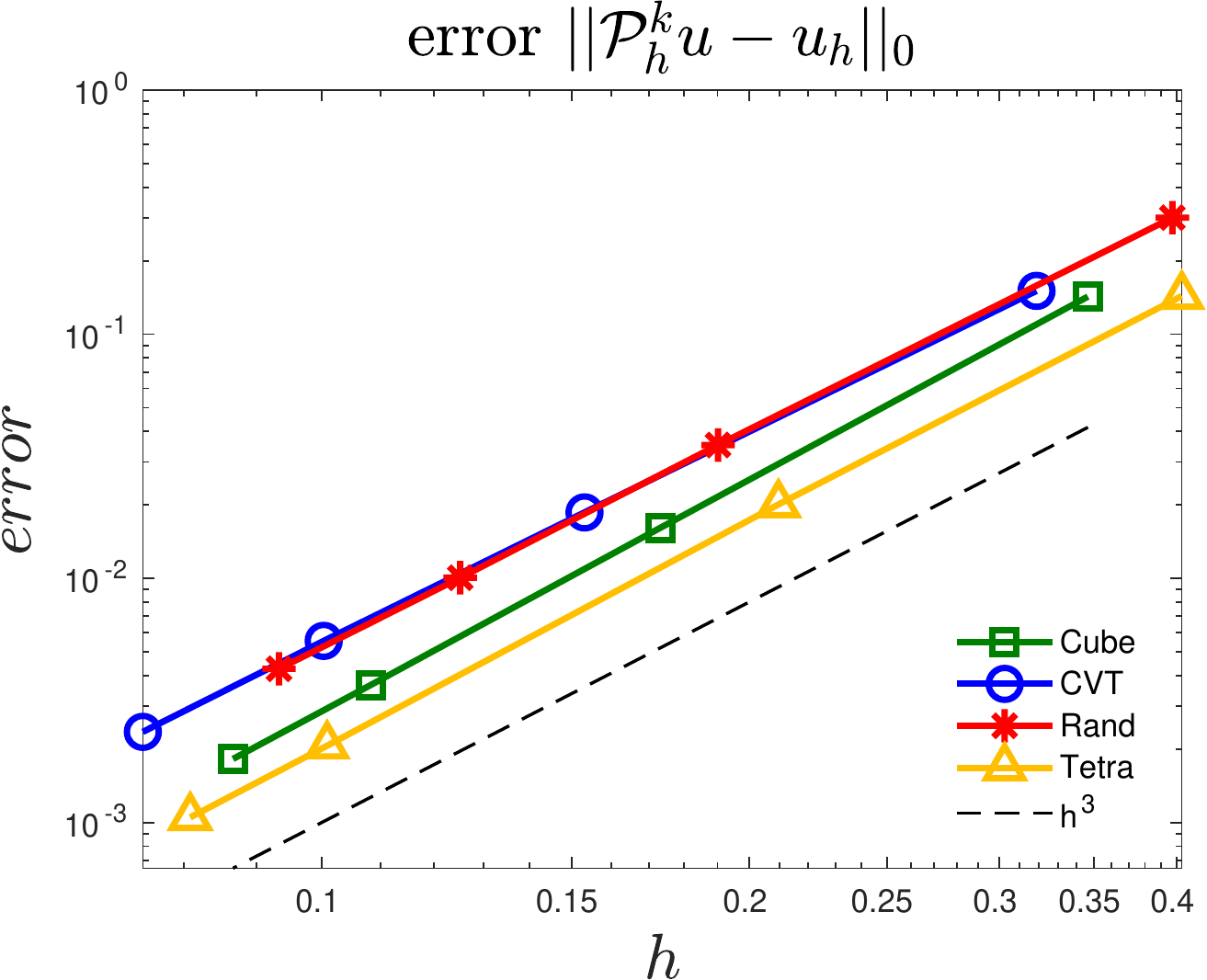}}
\subfigure[]{\includegraphics[width=\sizeGraph\textwidth,trim = 0mm 0mm 0mm 0mm, clip]{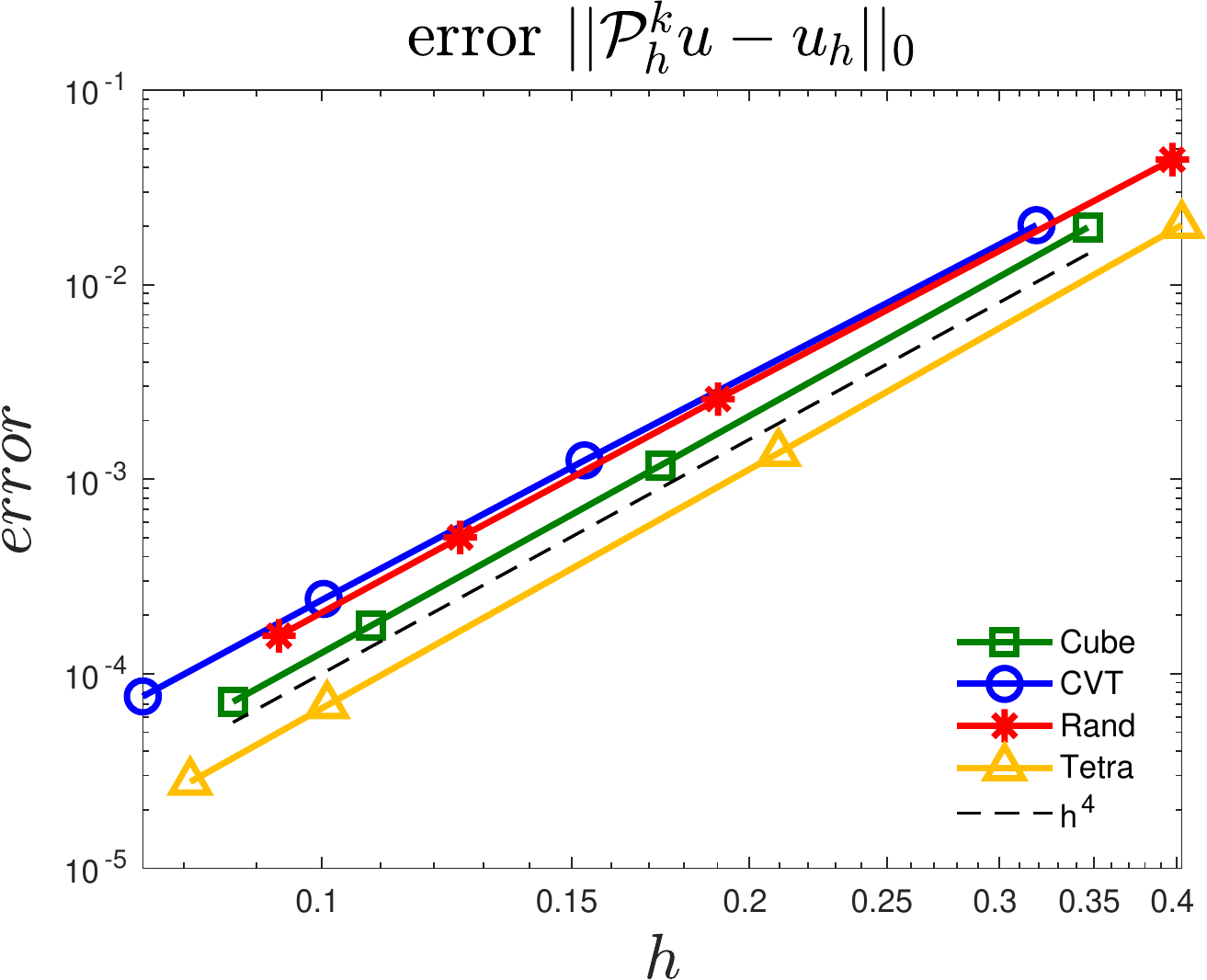}}\\
\caption{\textbf{Test a:} superconvergence results for all meshes, $k=1$ on the left and $k=2$ on the right.\label{fig:resuTest5}}	
\end{figure}
\begin{figure}[ht]
\centering
\subfigure[]{\includegraphics[width=\sizeGraph\textwidth,trim = 0mm 0mm 0mm 0mm, clip]{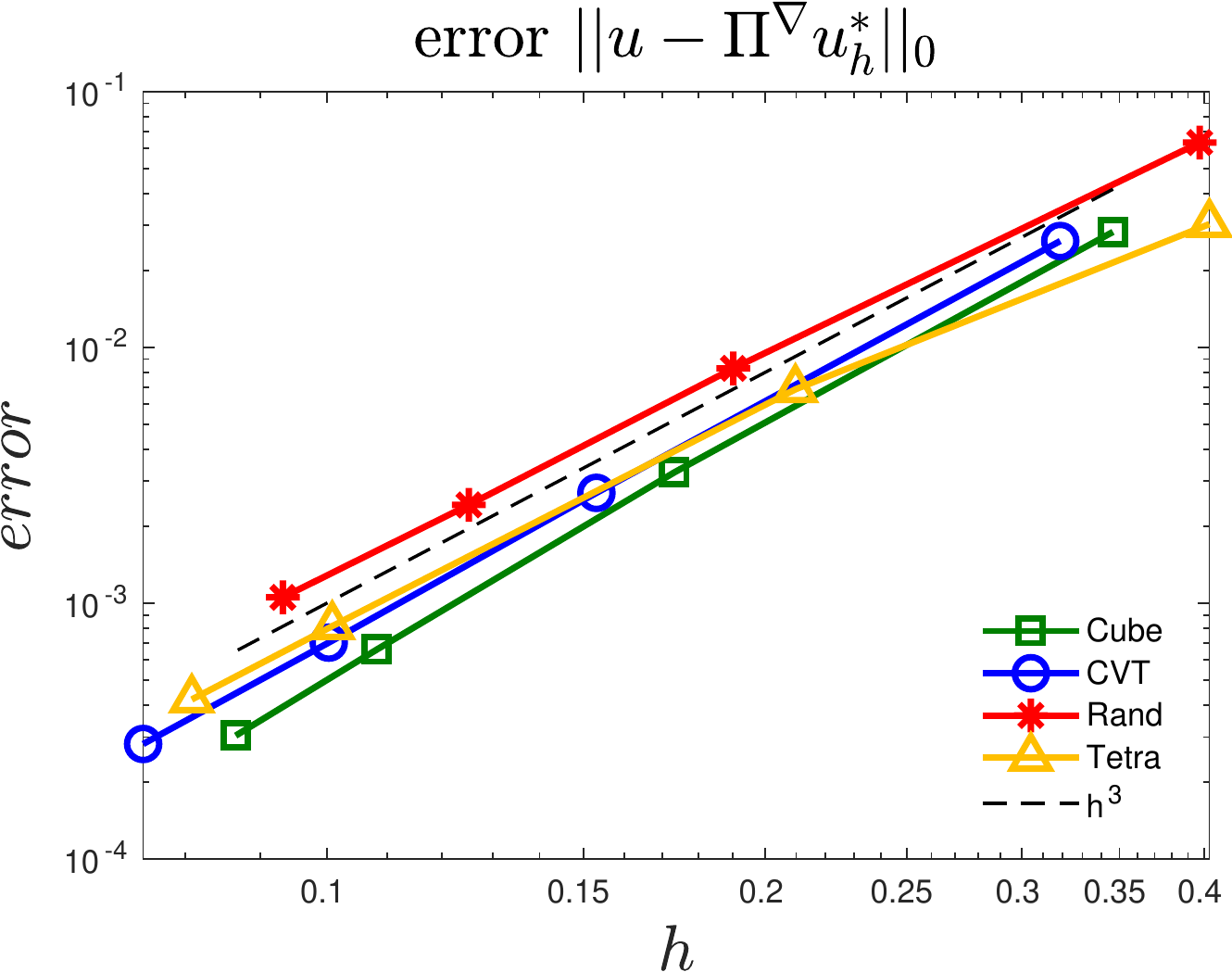}}
\subfigure[]{\includegraphics[width=0.33\textwidth,trim = 0mm 0mm 0mm 0mm, clip]{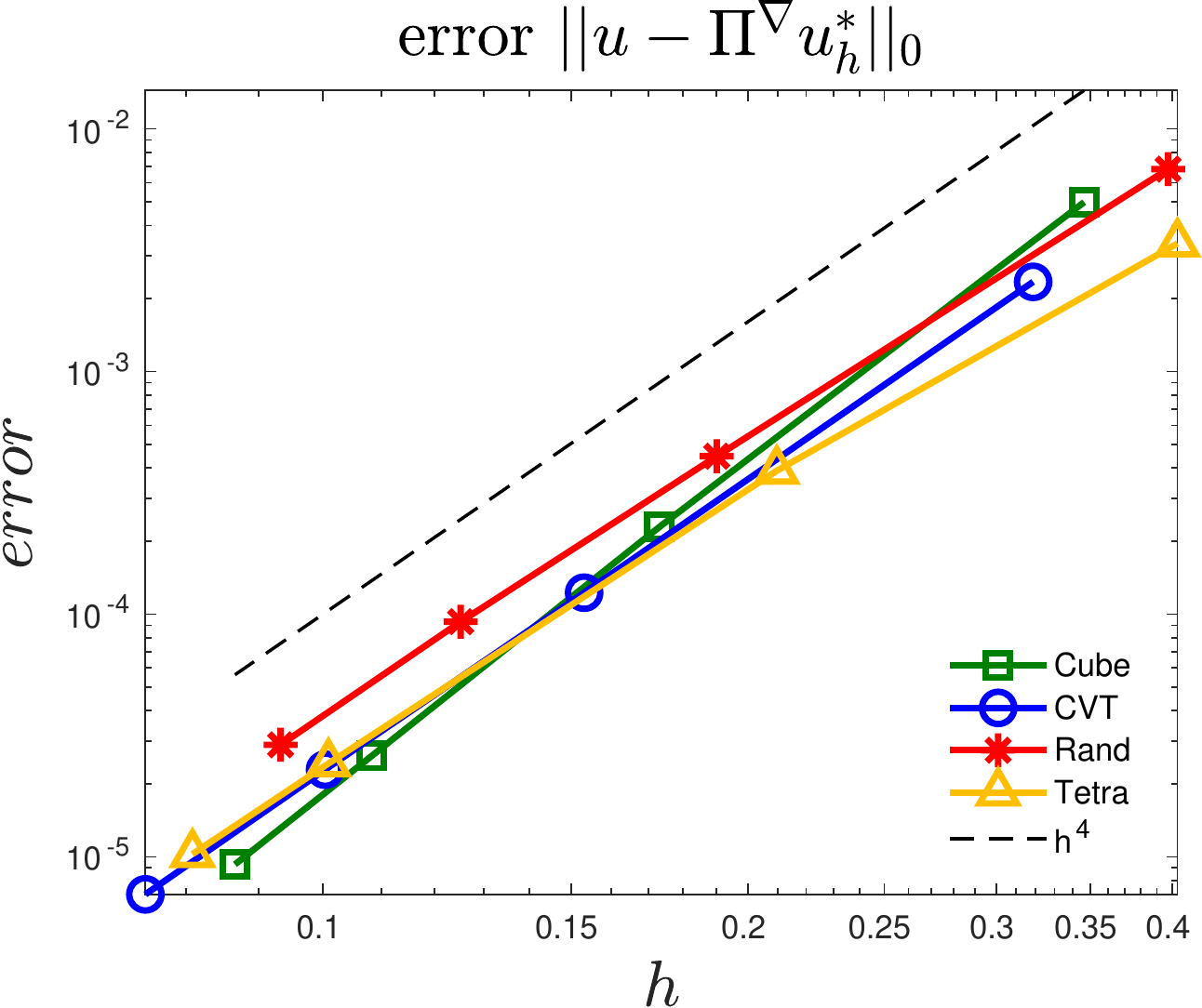}}\\
\caption{\textbf{Test a:} $h$-convergence results of the post-processed displacement for all meshes, $k=1$ on the left and $k=2$ on the right.\label{fig:resuTest6}}	
\end{figure}

In Figures~\ref{fig:resuTest5} and~\ref{fig:resuTest6} we report the convergence lines for the errors $E_{\P_h^k\bbu}$ and $E_{\uhstar}$, respectively. As expected, the asymptotic convergence rate is approximately equal to 3 when the degree of accuracy $k$ is 1, while the rate is 4 if we consider $k=2$. Moreover, also in this case the convergence lines are close to each others and this fact further confirms the robustness of the proposed scheme with respect to element shape.
	\section{Conclusion}
We have proposed a family of Virtual Element Methods for 3D linear elasticity problems described by the Hellinger-Reissner variational principle. The discrete stress tensors are a-priori symmetric, while the corresponding tractions are continuous across the element interfaces. The convergence and stability analysis has been confirmed by some numerical results.  Moreover, exploiting the hybridization procedure with the extra information derived from the original discrete solution and the Lagrange multipliers, we have achieved a better approximation for the displacement field. A possible future development of the present paper may concern the design of schemes for curved elements.
\section*{Acknowledgments}
The author is a member of the INdAM-GNCS. The author kindly acknowledges partial financial support by the INdAM-GNCS Project 2022 CUP E55F2200027001 and by the projects PRIN 2017 (No. 201744KLJL) and PRIN 2020 (No. 20204LN5N5), funded by the Italian Ministry of Universities and Research (MUR).
	
\bibliographystyle{plain}
	\bibliography{biblioSistemata}	
		\end{document}